\documentclass[12pt]{article}
	\usepackage[fleqn]{amsmath}
	\usepackage[cp1251]{inputenc}
	\usepackage[english,russian]{babel}
	\usepackage{amsmath,amsfonts,amssymb,graphicx,makeidx,amsthm}
	\usepackage{amscd,amsfonts,amssymb,multicol}
	\usepackage[T2A]{fontenc}
	\usepackage{euscript,mathrsfs}
	\usepackage{cite}
	\newtheorem{teo}{Теорема}
	\newtheorem{lem}{Лемма}

	\newtheorem{dfn}{Определение}
	
	\newtheorem{exampl}{Пример}
	\newtheorem{rem}{Замечание}
	
\begin{document}
\begin{center}		
	{\bf Решение обыкновенного дифференциального уравнения с дробной степенью оператора Бесселя} \\
	{\it Ситник С.М., Шишкина Э.Л.}
\end{center}

		\section{Введение}

		В этой статье рассмотрим дифференциальное уравнение с дробной степенью оператора Бесселя, где оператор Бесселя имеет вид
		\begin{equation}\label{Bess}
			B_\gamma= \frac{d^2}{dx^2}+\frac{\gamma}{x}\frac{d}{dx},\qquad \gamma\geq 0.
		\end{equation}
	Первые явные формулы для дробных степеней оператора Бесселя на отрезке в терминах гипергеометрических функций Гаусса приведены в статье Иды Шпринхайзен-Купер \cite{Ida}.
Более подробная теория дробных степеней \eqref{Bess} на отрезке и полуоси содержится в 
 \cite{Sita1,SitSh3,SitSh1,SitSh2,SitSh4}.
Дробные степени гипер-бесселевого оператора вида
		$$
		\mathbf{B}_{\alpha_0,\alpha_1,...,\alpha_m}=x^{\alpha_0}\frac{d}{dx}x^{\alpha_1}\frac{d}{dx}...x^{\alpha_{m-1}}\frac{d}{dx}x^{\alpha_m}
		$$
	с вещественными параметрами $\alpha_0,...,\alpha_m$ были представлены Адамом Макбрайдом в \cite{McBArt}. Их изучение было продолжено в \cite{Dim66,Dim68,DimKir,Kir1}.
	Оператор Бесселя   \eqref{Bess} соответствует  $\mathbf{B}_{\alpha_0,\alpha_1,...,\alpha_m}$ при
		$$
		m=2,\, \alpha_0=-1,\, \alpha_1=2-\gamma,\, \alpha_2=\gamma-1,
		$$
	или 
		$$
		m=2,\, \alpha_0=-\gamma,\, \alpha_1=\gamma,\, \alpha_2=0.
		$$

		Уравнения с дробными производными Бесселя ранее не изучались из-за отсутствия подходящих инструментов для их изучения.
		Первой целью статьи является представление одного из таких инструментов, а именно интегрального преобразования Мейера. Такое преобразование играет ту же роль для левосторонней дробной производной Бесселя на полуоси, что и преобразование Лапласа для левой дробной производной Герасимова-Капуто на полуоси. Другая цель состоит в том, чтобы показать, что степенные функции, умноженные на функции Фокса-Райта, являются фундаментальной системой решений левосторонней дробной производной Бесселя типа Герасимова-Капуто на полуоси.
		Уравнения с дробными производными Бесселя чрезвычайно интересны с теоретической точки зрения, но возникают и в приложениях, например, в задачах случайного блуждания частицы
		\cite{Garra1,Garra2}.

		В \cite{Kilbas}, на стр. 312 метод преобразования Лапласа был применен для получения явного решения однородного уравнения вида
		\begin{equation}\label{EqCap}
			(\,^{GC} D^\alpha_{0+}f)(x)=\lambda f(x),\qquad x>0,\qquad l-1<\alpha\leq l,\qquad l\in\mathbb{N},\qquad \lambda\in\mathbb{R},
		\end{equation}
		где для нецелого $\alpha>0$
		\begin{equation}\label{Cap}
			(\,^{GC} D^\alpha_{0+}f)(x)=\frac{1}{\Gamma(n-\alpha)}\int\limits_0^x\frac{f^{(n)}(t)dt}{(x-t)^{\alpha+1-n}},\qquad x\in[0,\infty)
		\end{equation}
	--- левосторонняя дробная производная Герасимова-Капуто на полуоси	 (\cite{Gerasimov}, \cite{Kilbas}, стр. 97, формула 2.4.47).
Для $\alpha=n=0,1,2,...$
		$$
		(\,^{GC} D^n_{0+}f)(x)=f^{(n)}(x).
		$$
		Герасимов в \cite{Gerasimov} вывел и решил уравнения с частными производными дробного порядка с производной \eqref{Cap} для прикладных задач механики в 1948 году.

		Условия вида
		\begin{equation}\label{ConCap}
			f^{k}(0+)=d_k,\qquad k=0,1,...,l-1,\qquad d_k\in\mathbb{R}
		\end{equation}
		могут быть добавлены к уравнению \eqref{EqCap}.
	Решение задачи  \eqref{EqCap}--\eqref{ConCap} имеет вид (см. \cite{Kilbas}, стр. 312)
		\begin{equation}\label{SolCap}
			f(x)=\sum\limits_{k=0}^{l-1}d_k\,x^k\,E_{\alpha,k+1}(\lambda x^\alpha),
		\end{equation}
		где $E_{\alpha,\beta}$ --- функция Миттаг--Леффлера \eqref{ML}.
		
		В этой статье мы, при помощи преобразования Мейера, получим точное решение  $f$
		однородного уравнения вида
		$$
		(\mathcal{B}_{\gamma,0+}^{\alpha} f)(x)=\lambda f(x),
		$$
		где положительная вещественная степень оператора \eqref{Bess} определена формулой \eqref{DrobessDer2}.

		\section{Основные определения}
	
		\subsection{Специальные функции}
	
	Приведем определения специальных функций, которые будем использовать.
	
	\textbf{Модифицированные функции Бесселя} (или гиперболические функции Бесселя) \textbf{первого и второго рода} $I_{\alpha }(x)$ и $K_{\alpha }(x)$  определяются формулами (см. \cite{Watson,Bowman,Kreh,Luke}) 
	\begin{equation}\label{Mod1}
		I_{\alpha }(x)=i^{-\alpha }J_{\alpha }(ix)=\sum _{m=0}^{\infty }{\frac {1}{m!\,\Gamma (m+\alpha +1)}}\left({\frac {x}{2}}\right)^{2m+\alpha },
	\end{equation}
	\begin{equation}\label{Mod2}
		K_{\alpha }(x)={\frac {\pi }{2}}{\frac {I_{-\alpha }(x)-I_{\alpha }(x)}{\sin(\alpha \pi )}},
	\end{equation}
	где $\alpha$ --- нецелое. Для целого $\alpha$ используется предельный переход в \eqref{Mod1},\eqref{Mod2}.
	Очевидно, что  $K_{\alpha }(x)=K_{-\alpha }(x)$.
	Для мальньких значений аргумента $0 < |r|\ll\sqrt{\nu+1}$, имеем
	\begin{equation}\label{AssimK}
		K_{\nu }(r)\sim {\begin{cases}-\ln \left({\dfrac {r}{2}}\right)-\vartheta \,\,\, {\text{if }}\nu =0,\\\frac {\Gamma (\nu )}{2^{1-\nu}}r^{-\nu}\,\,\, {\text{if }}\nu >0,\end{cases}}
	\end{equation}
	где
	$$
	\vartheta=\lim _{n\to \infty }\left(-\ln n+\sum _{k=1}^{n}{\frac {1}{k}}\right)=\int\limits_{1}^{\infty }\left(-{\frac {1}{x}}+{\frac {1}{\lfloor x\rfloor }}\right)\,dx
	$$ --- постоянная Эйлера-Маскерони \cite{abramovic}.
	
	Ядром преобразования Мейера является 
	 \textbf{нормированная модифицированная функция Бесселя второго рода}  $k_\nu$, определенная формулой 
	\begin{equation}\label{FBess3}
		k_\nu(x) =\frac{2^\nu\Gamma(\nu+1)}{x^\nu}\,\,K_\nu(x),
	\end{equation}
	где  $K_\nu$ --- модифицированная функция Бесселя второго рода  \eqref{Mod2}.

	Нормированная модифицированная функция Бесселя второго рода обладает свойствами
	\begin{equation}\label{KBessAsO2}
		\lim\limits_{x\rightarrow 0}k_\nu(x)=\frac{\Gamma(-\nu)}{2^{2\nu+1}\Gamma(1+\nu)},\qquad \nu<0,\qquad -\nu\notin\mathbb{N},
	\end{equation}
	\begin{equation}\label{KBessAsO3}
		\lim\limits_{x\rightarrow 0}x^\alpha k_0(x)=0,\qquad \alpha>0,\qquad 	\lim\limits_{x\rightarrow 0}\frac{1}{\ln{x}} k_0(x)=-1,
	\end{equation}
	\begin{equation}\label{KBessAsO1}
		\lim\limits_{x\rightarrow 0}x^{2\nu}k_\nu(x)=\frac{1}{2\nu},\qquad \nu>0,
	\end{equation}
	\begin{equation}\label{KBessAsDer}
		\lim\limits_{x\rightarrow 0}x^{2\nu+1}\frac{dk_\nu(x)}{dx}=-1,\qquad \nu>-1.
	\end{equation}

	Ядром левосторонней дробной производной Бесселя на полуоси является \textbf{гипергеометрическая функция Гаусса}, которая  внутри круга $ | z | {<} 1 $, определяется как сумма гипергеометрического ряда (см. \cite{abramovic}, стр. 373, формула 15.3.1)
	\begin{equation}\label{FG}
		\,_2F_1(a,b;c;z)=F(a,b,c;z)=\sum\limits_{k=0}^\infty\frac{(a)_k(b)_k}{(c)_k}\frac{z^k}{k!},
	\end{equation}
а для $ |z|\geq 1 $ получается аналитическим продолжением этого ряда.
	В (\ref{FG}) параметры $a,b,c$ и переменная  $z$ могут быть комплексными, $c{\neq}0,{-}1,{-}2,{\dots}$. Множитель
	$(a)_k$
	--- это символ Похгаммера:
	  $(z)_n=z(z+1)...(z+n-1),$ $n=1,2,...,$ $(z)_0\equiv 1$.
	
	 \textbf{Функция Миттаг--Леффлёра} $E_{\alpha,\beta}(z)$ --- это целая функция порядка $1/\alpha$
	 определяется следующим рядом, в случае когда вещественная часть $ \alpha $ строго положительна
	\begin{equation}\label{ML}
		E_{\alpha, \beta}(z)=\sum_{n=0}^{\infty}\frac{z^n}{\Gamma(\alpha n+\beta)},\;\; z\in\mathbb{C},\;\alpha, \beta\in\mathbb{C},\;	 {\rm Re}\,\alpha>0,\; {\rm Re}\,\beta>0.
	\end{equation}
Функция \eqref{ML} была введена Гестой Миттаг--Леффлёром в 1903 году для $ \alpha {=} 1 $ и А.~Виманом в 1905 году в общем случае.
Первыми приложениями этих функций Миттаг--Леффлёра и Вимана были приложения в комплексном анализе (нетривиальные примеры целых функций с нецелыми порядками роста и обобщенные методы суммирования). В СССР эти функции стали в основном известны после публикации знаменитой монографии М. М. Джрбашяна \cite{Dzh} (см. также его более позднюю монографию \cite{Djr2}). 
Наиболее известным применением функций Миттаг--Леффлёра в теории интегро-дифференциальных уравнений и дробного исчисления является тот факт, что  резольвента дробного интеграла Римана--Лиувилля явно выражается через них в соответствии со знаменитой  формулой Хилле--Тамаркина--Джрбашяна 
\cite{SKM}, стр. 78.
Ввиду многочисленных приложений к решению дифференциальных уравнений с дробными производными эта функция была заслуженно названа
в \cite{GoMa} была заслуженно названа "\textit{Королевской функцией дробного исчисления}".

	 \textbf{Функция Фокса-Райта} $_p\Psi_q(z)$ для $z\in \mathbb{C}$, $a_l,b_j\in\mathbb{C}$,
	$\alpha_l,\beta_j\in\mathbb{R}$, $l=1,...,p;$ $j=1,...,q$  определяется рядом вида (см. \cite{Fox,Wright})
	\begin{equation}\label{Wright}
		\,_p\Psi_q(z)=\,_p\Psi_q\left[\left.
		\begin{array}{c}
			$$(a_l,\alpha_l)_{1,p}$$ \\
			$$(b_j,\beta_j)_{1,q}$$ \\
		\end{array}
		\right|z\right]=
		\sum\limits_{k=0}^\infty \frac{\prod\limits_{l=1}^p\Gamma(a_l+\alpha_l k)}{\prod\limits_{j=1}^q\Gamma(b_j+\beta_j k)}
		\frac{z^k}{k!}.
	\end{equation}
При условии
	$$
	\sum\limits_{j=1}^q\beta_j-\sum\limits_{l=1}^p\alpha_l>-1
	$$
ряд в \eqref{Wright} сходится для всех $z\in\mathbb{C}$.
Пусть
	$$
	\delta=\prod\limits_{l=1}^p|\alpha_l|^{-\alpha_l}\prod\limits_{j=1}^q|\beta_j|^{\beta_j},
	$$
	$$
	\mu=\sum\limits_{j=1}^q b_j-\sum\limits_{l=1}^pa_l+\frac{p-q}{2}.
	$$
	Если
	$$
	\sum\limits_{j=1}^q\beta_j-\sum\limits_{l=1}^p\alpha_l=-1,
	$$
то ряд	в \eqref{Wright} сходится абсолютно для $|z|<\delta$ и для $|z|=\delta$, когда  ${\rm Re}\,\mu>\frac{1}{2}$.
		Функция Фокса-Райта  для дробных степеней оператора Бесселя играет ту же роль, что функция Миттаг-Леффлера для обыкновенного дробного исчисления.

	Используя функцию Фокса-Райта \eqref{Wright}, мы можем записать функцию Миттаг--Леффлёра в виде
	\begin{equation}\label{FWML}
		E_{\alpha, \beta}(z)=\,_1\Psi_1\left[\left.
		\begin{array}{c}
			$$(1,1)$$ \\
			$$(\beta,\alpha)$$ \\
		\end{array}
		\right|z\right].
	\end{equation}

	\subsection{Интегральные преобразования и оператор Пуассона}

	В этом пункте мы представляем интегральные преобразования Лапласа и Мейера и их связь с оператором преобразования  Пуассона.
	
	 \textbf{Преобразование Лапласа} функции $f(t)$, определенной для всех вещественных чисел  $t>0$, --- это функция $F(s)$, представимая равенством
	\begin{equation}\label{Lap}
		\mathcal{L}[f](s)=F(s)=\int\limits_{0}^{\infty }f(t)e^{-st}\,dt,
	\end{equation}
где  $s$ --- комплексное число
	$ s=\sigma +i\omega$, $\sigma$ и $\omega$ --- вещественные.
	
	Пусть $\mathscr{E}_a$, $a\in\mathbb{R}$ --- пространство функций $f:\mathbb{R}\rightarrow\mathbb{C}$, $f\in L_1^{loc}(\mathbb{R})$, таких, что 
	$\int\limits_0^\infty |f(t)|e^{-at}dt<\infty$ и $f(t)$ обращается в нуль, если  $t<0$.
	
	Пусть $f\in \mathscr{E}_a$. Тогда интеграл Лапласа \eqref{Lap} сходится абсолютно и равномерно на $\bar{H}_a = \{p : p\in\mathbb{C}, {\rm Re}\,p\geq a\}$.
Преобразование Лапласа функции  $f\in \mathscr{E}_a$ ограничено на $\bar{H}_a$
и является аналитической функцией на ${H}_a = \{p : p\in\mathbb{C}, {\rm Re}\,p> a\}$ (см. \cite{Glaeske}, стр. 28).
	
	Пусть  $f\in \mathscr{E}_a$ и гладкая на каждом интервале  $(a, b)\in\mathbb{R}_+$.  Тогда в точках  $t$
непрерывности этой функции определено обратное преобразование Лапласа:
	$$
	\mathcal{L}^{-1}[F](t)=f(t)=\frac{1}{2\pi i}\int\limits_{c-i\infty}^{c+i\infty}F(s)e^{ts}ds,\qquad c>a.
	$$
(см. \cite{Glaeske}, стр. 37).

	Преобразование Лапласа функции Миттаг--Леффлера, умноженной на степенную функцию имеет вид (см. \cite{Kilbas}, стр. 47, формула 1.9.13, где $\rho=1$) 
	\begin{equation}\label{LapOfML}
		\mathcal{L} [x^{\beta-1}E_{\alpha,\beta}(\lambda x^\alpha)](s)=\frac{s^{\alpha-\beta}}{s^\alpha-\lambda}.
	\end{equation}

Для функции $f$ интегральное преобразование, содержащее функцию $k_{\frac{\gamma-1}{2}}$, $\gamma\geq 1$ в качестве ядра называется  \textbf{преобразованием Мейера}. Оно определяется формулой
	\begin{equation}\label{MeT}
		\mathcal{K}_\gamma[f](\xi)=F(\xi)=\int\limits_{0}^{\infty} {k}_{\frac{\gamma-1}{2}} (x\xi)\,
		f(x)x^\gamma\,dx.
	\end{equation}
	
	Преобразование \eqref{MeT} --- это модификация  $K$--преобразования из \cite{Glaeske}, стр. 93, формула 1.8.48 и поэтому имеет те же свойства, но другое асимптотическое поведение.

	
	Пусть $\beta$ такое число, что
	$$
	\beta>\frac{\gamma}{2}-2\, {\text{если}}\,\gamma>1\, {\text и}\,\beta>-1 \,{\text{если}}\,\gamma= 1\, {\text и}\,\ \beta>-1-\frac{\gamma}{2}  \,{\text{если}}\, 0<\gamma<1.
	$$
	
	Определим класс функций
	$$
	\mathscr{M}_{\gamma}^a(\mathbb{R}_+)=\biggl\{f\in L^{loc}_1(\mathbb{R}_+):f(t) = o\left(t^{\beta-\frac{\gamma}{2}}\right)\,\text{при}\, t\rightarrow+0\,\text{и}\,f(t) = O(e^{at})\,\text{при}\,t\rightarrow+\infty\biggr\}.
	$$
Преобразование Мейера функции $f\in \mathscr{M}_{\gamma}^a(\mathbb{R}_+)$ существует почти всюду для  ${\rm Re}\,\xi > a$ (см. \cite{Glaeske}, стр. 94).

	Если $0<\gamma<2$ и $F(\xi)$ --- аналитическая на
полуплоскости $H_a=\{p\in\mathbb{C}:{\rm Re}\,p\geq a\}$, $a\leq 0$ и
 $s^{\frac{\gamma}{2}-1}F(\xi)\rightarrow 0$, $|\xi|\rightarrow+\infty$ равномерно по ${\rm arg}\,s$, то для любого числа $c$, $c > a$
	обратное $\mathcal{K}_\gamma^{-1}$ имеет вид (см. \cite{Glaeske}, стр. 94)
	\begin{equation}\label{InvMeT}
		\mathcal{K}_\gamma^{-1}[\widehat{f}](x)=f(x)=\frac{1}{\pi i}\int\limits_{c-i\infty}^{c+i\infty} \widehat{f}(\xi) i_{\frac{\gamma-1}{2}}(x\xi) \xi^\gamma d\xi.
	\end{equation}
	
Формула обращения \eqref{InvMeT} 
не удобна для расчетов и имеет условие $ 0{<}\gamma{<}2 $. 
Здесь мы представим другую формулу обращения с использованием  оператора преобразования Пуассона.

Пусть $\gamma{>}0$.	Одномерный оператор Пуассона определен для интегрируемых функций  $f$ равенством
	\begin{equation}
		\label{154}
		\mathcal{P}_x^{\gamma}f(x)=\frac{2C(\gamma)}{x^{\gamma-1}}
		\int\limits_0^x \left( x^2-t^2\right)^{\frac{\gamma}{2}-1}f(t)\,dt,\qquad   C(\gamma)=	\frac{\Gamma\left(\frac{\gamma+1}{2}\right)}{\sqrt{\pi}\,\Gamma\left(\frac{\gamma}{2}\right)}.
	\end{equation}
Постоянная $C(\gamma)$ выбрана так, чтобы  $\mathcal{P}_x^\gamma[1]=1$ (см. \cite{BookSSh}, стр. 50).
	
	Левый обратный оператор для \eqref{154} при $\gamma>0$ для  функции  $H(x)$ определяется формулой (см. \cite{SKM})
	\begin{equation}\label{ObrPuass}
		(\mathcal{P}^\gamma_x)^{-1} H(x)=\frac{2\sqrt{\pi}x}{\Gamma\left(\frac{\gamma+1}{2}\right)\Gamma \left( n-\frac{\gamma}{2} \right) } \left(\frac{d}{2xdx} \right)^{n} \int\limits_{0}^{x} H(z)  (x^2-z^2)^{n-\frac{\gamma}{2}-1}z^{\gamma} dz,
	\end{equation}
	где $
	n=\left[\frac{\gamma}{2}\right]+1.
	$

	Для того чтобы найти
	$f(x)$ из равенства
	$$
	\mathcal{K}_\gamma[f](\xi)=(\mathcal{L}F(z))(\xi)=g(\xi)
	$$
	представим ядро преобразования \eqref{MeT} по формуле
	$$
	K_\alpha(x\xi)=\frac{\Gamma\left(\frac{1}{2} \right) }{\Gamma\left(\alpha+\frac{1}{2} \right)}\left(\frac{x\xi}{2} \right)^\alpha
	\int\limits_{1}^\infty e^{-x\xi t}(t^2-1)^{\alpha-\frac{1}{2}}dt=\{x t=z\}=
	$$
	$$
	=\frac{\sqrt{\pi} }{\Gamma\left(\alpha+\frac{1}{2} \right)}\left(\frac{\xi}{2x} \right)^\alpha
	\int\limits_{x}^\infty e^{-\xi z}(z^2-x^2)^{\alpha-\frac{1}{2}}dz
	$$
	из \cite{Watson},	стр. 190, формула (4). Тогда
	$$
	{k}_{\frac{\gamma-1}{2}} (x\xi)=\frac{2^{\frac{1-\gamma}{2}} }{\Gamma\left( \frac{\gamma+1}{2}\right)(x\xi)^{\frac{\gamma-1}{2}}}K_{\frac{\gamma-1}{2}}(x\xi)=
	$$
	$$
	=\frac{2^{\frac{1-\gamma}{2}} }{\Gamma\left( \frac{\gamma+1}{2}\right)(x\xi)^{\frac{\gamma-1}{2}}}\frac{\sqrt{\pi} }{\Gamma\left(\frac{\gamma}{2} \right)}\left(\frac{\xi}{2x} \right)^{\frac{\gamma-1}{2}}
	\int\limits_{x}^\infty e^{-\xi z}(z^2-x^2)^{\frac{\gamma}{2}-1}dz=
	$$
	$$
	=\frac{2^{1-\gamma}\sqrt{\pi} }{x^{\gamma-1}\Gamma\left( \frac{\gamma+1}{2}\right)\Gamma\left(\frac{\gamma}{2} \right)}
	\int\limits_{x}^\infty e^{-\xi z}(z^2-x^2)^{\frac{\gamma}{2}-1}dz.
	$$
	Следовательно,
	$$
	\mathcal{K}_\gamma[f](\xi)=\widehat{f}(\xi)=\int\limits_{0}^{\infty} {k}_{\frac{\gamma-1}{2}} (x\xi)\,
	f(x)x^\gamma\,dx=
	$$
	$$
	=\frac{2^{1-\gamma}\sqrt{\pi} }{\Gamma\left( \frac{\gamma+1}{2}\right)\Gamma\left(\frac{\gamma}{2} \right)}\int\limits_{0}^{\infty}
	f(x)x\,dx\int\limits_{x}^\infty e^{-\xi z}(z^2-x^2)^{\frac{\gamma}{2}-1}dz=
	$$
	$$
	=\frac{2^{1-\gamma}\sqrt{\pi} }{\Gamma\left( \frac{\gamma+1}{2}\right)\Gamma\left(\frac{\gamma}{2} \right)}\int\limits_{0}^{\infty}
	e^{-\xi z}dz \int\limits_{0}^z
	f(x)(z^2-x^2)^{\frac{\gamma}{2}-1}xdx.
	$$
Используя оператор Пуассона \eqref{154} и преобразование Лапласа \eqref{Lap}, получим
	$$
	\mathcal{K}_\gamma[f](\xi)
	=\int\limits_{0}^{\infty}
	e^{-\xi z}F(z)dz=(\mathcal{L}F(z))(\xi),
	$$
где
	$$
	F(z)=A_\gamma z^{\gamma-1}\mathcal{P}_z^\gamma zf(z),\qquad A_\gamma=\frac{\pi}{2^\gamma \Gamma^2\left(\frac{\gamma+1}{2} \right) }.
	$$
Таким образом, для того чтобы найти  $f(x)$ из равенства
	$$
	\mathcal{K}_\gamma[f](\xi)=(\mathcal{L}A_\gamma z^{\gamma-1}\mathcal{P}_z^\gamma zf(z))(\xi)=g(\xi)
	$$
	мы должны сначала обратить преобразование Лапласа, а затем обратить оператор Пуассона. Формула обращения для функции $g$, такой что  $(\mathcal{L}^{-1}g)(x)$
существует, имеет вид    
	\begin{equation}\label{InvK}
		f(x)=\mathcal{K}_\gamma^{-1}[g](x)=\frac{1}{A_\gamma x}(\mathcal{P}_x^\gamma)^{-1}x^{1-\gamma}(\mathcal{L}^{-1}g)(x),\quad g=	\mathcal{K}_\gamma[f],\quad A_\gamma=\frac{\pi}{2^\gamma \Gamma^2\left(\frac{\gamma+1}{2} \right) }.
	\end{equation}

		\section{Левосторонние дробные интегралы и производные Бесселя на полуоси}
	
	\subsection{Определения левосторонних дробных интегралов и производных Бесселя на полуоси}

	Пусть $\alpha>0$, $\gamma>0$.	
	 \textbf{Левосторонний дробный интеграл Бесселя на полуоси} $B_{\gamma,0+}^{-\alpha}$  для $f{\in}L[0,\infty)$ определяется формулой
	$$
	(B_{\gamma,0+}^{-\alpha}f)(x)=(IB_{\gamma,0+}^{\alpha}\,f)(x)=
	$$
	\begin{equation}\label{Bess4}
		=\frac{1}{\Gamma(2\alpha)}\int\limits_0^x\left(\frac{y}{x}\right)^\gamma\left(\frac{x^2{-}y^2}{2x}\right)^{2\alpha-1}\,_2F_1\left(\alpha{+}\frac{\gamma{-}1}{2},\alpha;2\alpha;1{-}\frac{y^2}{x^2}\right)f(y)dy.
	\end{equation}
	Для $\alpha<0$ 	формула \eqref{Bess4} может быть продолжена аналитически, а$(B_{\gamma,0+}^{0}f)(x)=f(x)$.

	В \cite{McBArt} были представлены пространства, адаптированные для работы с операторами вида $B_{\gamma,0+}^\alpha$,  $\alpha\in\mathbb{R}$. Эти пространства имеют вид:
	$$
	F_p=\left\{\varphi\in C^\infty(0,\infty):x^k\frac{d^k\varphi}{dx^k}\in L^p(0,\infty)\,{\text{for}}\,k=0,1,2,...\right\},\qquad 1\leq p<\infty,
	$$
	$$
	F_\infty=\left\{\varphi\in C^\infty(0,\infty):x^k\frac{d^k\varphi}{dx^k}\rightarrow0 \, {\text{as}}\, x\rightarrow0+
	\,{\text{and\,as}}\, x\rightarrow\infty\,{\text{for}}\,k=0,1,2,...\right\}
	$$
и
	$$
	F_{p,\mu}=\left\{\varphi: x^{-\mu}\varphi(x)\in F_p\right\},\qquad 1\leq p\leq \infty,\qquad \mu\in\mathbb{C}.
	$$
	
	Мы приведем здесь теорему, которая является частным случаем теорем из
 \cite{McBArt}.
	\begin{teo}\label{Teo1} Пусть $\alpha{\in}\mathbb{R}$. Для всех $p,\mu$ и $\gamma>0$ таких, что $\mu{\neq}\frac{1}{p}{-}2m$, $\gamma{\neq}\frac{1}{p}{-}\mu{-}2m{+}1$, $m{=}1,2{...}$ оператор $B_{\gamma,0+}^{\alpha}$ является непрерывным линейным отображением из  $F_p,\mu$ в $F_{p,\mu-2\alpha}$. Если, кроме того,  $2\alpha\neq\mu-\frac{1}{p}+2m$ и $\gamma-2\alpha\neq\frac{1}{p}-\mu-2m+1$, $m=1,2...$,
	то $B_{\gamma,0+}^{\alpha}$ гомеоморфизм из $F_p,\mu$ на $F_{p,\mu-2\alpha}$ с обратным оператором $B_{\gamma,0+}^{-\alpha}$.
	\end{teo}

Сравним дробный интеграл Бесселя $B_{\gamma,0+}^{-\alpha}$ с известным дробным интегралом Римана-Лиувилля $I_{0+}^{2\alpha}$. Для этого положим $\gamma=0$:
	$$
	(B_{0,0+}^{-\alpha}f)(x){=}\frac{1}{\Gamma(2\alpha)}\int\limits_a^x\left(\frac{x^2-y^2}{2x}\right)^{2\alpha-1}
	\,_2F_1\left(\alpha-\frac{1}{2},\alpha;2\alpha;1-\frac{y^2}{x^2}\right)f(y)dy{=}
	$$
	$$
	=\frac{1}{\Gamma(2\alpha)}\int\limits_0^x\left(\frac{x^2-y^2}{2x}\right)^{2\alpha-1}\left[\frac{2x}{x+y}\right]^{2\alpha-1}f(y)dy=
	$$
	$$
	=\frac{1}{\Gamma(2\alpha)}\int\limits_0^x (x-y)^{2\alpha-1} f(y)dy=(I_{0+}^{2\alpha}f)(x).
	$$

Теперь выпишем явную формулу для дробной производной Бесселя $B_\gamma^\alpha$, $\alpha>0$. Для приложений лучше использовать обобщение дробной производной Герасимова--Капуто \eqref{Cap}.
	
	\begin{dfn}\label{def1}
	Пусть  $n{=}[\alpha]{+}1$, $f{\in}L[0,\infty)$, $IB_{\gamma,b-}^{n-\alpha}f,IB_{\gamma,b-}^{n-\alpha}f{\in} C^{2n}_{ev}(0,\infty)$.   \textbf{Левосторонняя дробная производная Бесселя на полуоси типа Герасимова-Капуто} определяется равенством
	\begin{equation}\label{DrobessDer2}
		(\mathcal{B}_{\gamma,0+}^\alpha f)(x)=(IB_{\gamma,0+}^{n-\alpha}B_\gamma^nf)(x).
	\end{equation}
	Легко видеть, что
	$$
	(\mathcal{B}_{0,0+}^{\alpha}f)(x){=}(\,^{C}D^{2\alpha}_{0+}f)(x),
	$$
где $(\,^{GC}D^{2\alpha}_{0+}f)(x)$ определено формулой \eqref{Cap}.
\end{dfn}

	Следуя \cite{Ida} и \cite{McBArt} приведем следующие результаты.
Пусть ${\rm Re}\,(2\eta+\mu)+2>1/p$, и $\varphi\in F_{p,\mu}$. Для ${\rm Re}\,\alpha>0$, мы определим $I_2^{\eta,\alpha}\varphi$ формулой
	\begin{equation}\label{210}
		I_2^{\eta,\alpha}\varphi(x)=\frac{2}{\Gamma(\alpha)}\, x^{-2\eta-2\alpha}\int\limits_0^x(x^2-u^2)^{\alpha-1}u^{2\eta+1}\varphi(u)du.
	\end{equation}
	
	Выражение $I_2^{\eta,\alpha}$ продолжается на значения ${\rm Re}\,\alpha\leq0$ по формуле
	\begin{equation}\label{211}
		I_2^{\eta,\alpha}\varphi=(\eta+\alpha+1)I_2^{\eta,\alpha+1}\varphi+\frac{1}{2}I_2^{\eta,\alpha+1}\,x\frac{d\varphi}{dx}.
	\end{equation}

	\begin{teo} Для \eqref{Bess4} справедлива следующая факторизация
		\begin{equation}\label{F0+}
			(B_{\gamma,0+}^{-\alpha}\varphi)(x)=\left(\frac{x}{2}\right)^{2\alpha}\,I_2^{\frac{\gamma-1}{2},\alpha}I_2^{0,\alpha}\varphi,
		\end{equation}
		где
		$$
		I_2^{0,\alpha}\varphi(x)=\frac{2}{\Gamma(\alpha)}\, x^{-2\alpha}\int\limits_0^x(x^2-u^2)^{\alpha-1}u\varphi(u)du,
		$$
		$$
		I_2^{\frac{\gamma-1}{2},\alpha}\varphi(x)=\frac{2}{\Gamma(\alpha)}\, x^{1-\gamma-2\alpha}\int\limits_0^x(x^2-u^2)^{\alpha-1}u^{\gamma}\varphi(u)du.
		$$
	\end{teo}
	\begin{proof}
	Имеем
		$$
		(B_{\gamma,0+}^{-\alpha}\varphi)(x)=
		$$
		$$=\frac{1}{\Gamma(2\alpha)}\int\limits_0^x\left(\frac{u}{x}\right)^\gamma
		\left(\frac{x^2-u^2}{2x}\right)^{2\alpha-1}\,_2F_1\left(\alpha+\frac{\gamma-1}{2},\alpha;2\alpha;1-\frac{u^2}{x^2}\right)\varphi(u)du=
		$$
		$$
		=2^{-2\alpha}x^{2\alpha}I_2^{\frac{\gamma-1}{2},\alpha}I_2^{0,\alpha}\varphi=
		$$
		$$
		=\frac{2^{1-2\alpha}x^{2\alpha}}{\Gamma(\alpha)}I_2^{\frac{\gamma-1}{2},\alpha}y^{-2\alpha}\int\limits_0^y(y^2-u^2)^{\alpha-1}u\varphi(u)du=
		$$
		$$
		=\frac{2^{2-2\alpha}x^{2\alpha}}{\Gamma^2(\alpha)}x^{-\gamma+1-2\alpha}
		\int\limits_0^x(x^2-y^2)^{\alpha-1}y^{\gamma-2\alpha}dy\int\limits_0^y(y^2-u^2)^{\alpha-1}u\varphi(u)du=
		$$
		$$
		=\frac{2^{2-2\alpha}}{\Gamma^2(\alpha)}x^{1-\gamma}
		\int\limits_0^xu\varphi(u)du \int\limits_u^x(y^2-u^2)^{\alpha-1}(x^2-y^2)^{\alpha-1}y^{\gamma-2\alpha}dy.
		$$
		
		Найдем
		$$
		\int\limits_u^x(y^2-u^2)^{\alpha-1}(x^2-y^2)^{\alpha-1}y^{\gamma-2\alpha}dy=\{y^2=t\}=
		\frac{1}{2}\int\limits_{u^2}^{x^2}(t-u^2)^{\alpha-1}(x^2-t)^{\alpha-1}t^{\frac{\gamma-1}{2}-\alpha}dt=
		$$
		$$
		=\frac{\sqrt{\pi }\Gamma (\alpha )}{2^{2 \alpha }\Gamma
			\left(\alpha +\frac{1}{2}\right)}\,  \left(x^2-u^2\right)^{2\alpha -1}\,u^{-2 \alpha +\gamma -1} \,_2F_1\left(\alpha+\frac{1-\gamma }{2},\alpha;2 \alpha ;1-\frac{x^2}{u^2}\right).
		$$
		
		Используя формулу (см. \cite{abramovic})
		$$
		\,_2F_1(a,b;c;z)=(1-z)^{-a}\,_2F_1\left(a,c-b;c;\frac{z}{z-1}\right),
		$$
получим
		$$
		\,_2F_1\left(\alpha+\frac{1-\gamma }{2},\alpha;2 \alpha ;1-\frac{x^2}{u^2}\right)=
		\,_2F_1\left(\alpha,\alpha+\frac{1-\gamma }{2};2 \alpha ;1-\frac{x^2}{u^2}\right)=
		$$
		$$
		=\left(\frac{x^2}{u^2}\right)^{-\alpha}\,_2F_1\left(\alpha,\alpha+\frac{\gamma-1}{2};2 \alpha ;1-\frac{u^2}{x^2}\right)=$$
		$$	=\left(\frac{x^2}{u^2}\right)^{-\alpha}\,_2F_1\left(\alpha+\frac{\gamma-1}{2},\alpha;2 \alpha ;1-\frac{u^2}{x^2}\right)
		$$
	и
		$$
		\int\limits_u^x(y^2-u^2)^{\alpha-1}(x^2-y^2)^{\alpha-1}y^{\gamma-2\alpha}dy=\frac{\sqrt{\pi }\Gamma (\alpha )}{2^{2 \alpha }\Gamma
			\left(\alpha +\frac{1}{2}\right)}\times
		$$
		$$
		\times  \left(x^2-u^2\right)^{2\alpha -1}\,u^{-2 \alpha +\gamma -1}
		\left(\frac{x^2}{u^2}\right)^{-\alpha}\,_2F_1\left(\alpha,\alpha+\frac{\gamma-1}{2};2 \alpha ;1-\frac{u^2}{x^2}\right)=
		$$
		$$
		=\frac{\sqrt{\pi }\Gamma (\alpha )}{2^{2 \alpha }\Gamma
			\left(\alpha +\frac{1}{2}\right)}\,  \left(x^2-u^2\right)^{2\alpha -1}\,u^{\gamma -1}x^{-2 \alpha}\,_2F_1\left(\alpha,\alpha+\frac{\gamma-1}{2};2 \alpha ;1-\frac{u^2}{x^2}\right).
		$$
Наконец,
		$$
		(B_{\gamma,0+}^{-\alpha}\varphi)(x)
		=\frac{2^{2(1-2\alpha)}\sqrt{\pi }}{\Gamma(\alpha)\Gamma
			\left(\alpha +\frac{1}{2}\right)}\,x^{1-\gamma-2\alpha}\times
		$$
		$$
		\times	\int\limits_0^x\,  \left(x^2-u^2\right)^{2\alpha -1}\,u^{\gamma}\,_2F_1\left(\alpha+\frac{\gamma-1}{2},\alpha;2 \alpha ;1-\frac{u^2}{x^2}\right)\varphi(u)du.
		$$
	Применяя формулу удвоения вида
		$$
		\Gamma(\alpha)\Gamma
		\left(\alpha +\frac{1}{2}\right)=2^{1-2\alpha}\sqrt{\pi}\Gamma(2\alpha),
		$$
	получим
		$$
		(B_{\gamma,0+}^{-\alpha}\varphi)(x)
		=\frac{2^{1-2\alpha}}{\Gamma(2\alpha)}\,x^{1-\gamma-2\alpha}\times
		$$
		$$
		\times
		\int\limits_0^x\,  \left(x^2-u^2\right)^{2\alpha -1}\,u^{\gamma}\,_2F_1\left(\alpha+\frac{\gamma-1}{2},\alpha;2 \alpha ;1-\frac{u^2}{x^2}\right)\varphi(u)du=
		$$
		$$
		=\frac{1}{\Gamma(2\alpha)}\int\limits_0^x\,  \left(\frac{x^2-u^2}{2x}\right)^{2\alpha -1}\,\left(\frac{u}{x}\right)^{\gamma}\,_2F_1\left(\alpha+\frac{\gamma-1}{2},\alpha;2 \alpha ;1-\frac{u^2}{x^2}\right)\varphi(u)du.
		$$
	Что и дает \eqref{F0+}.
		Доказательство закончено.
	\end{proof}

		\subsection{Преобразование Мейера левосторонных дробных интегралов и производных Бесселя на полуоси}

	В этом разделе мы применим преобразование Мейера к левосторонним дробным интегралам и производным Бесселя на полуоси, а затем, в разделе \ref{Sec4}, мы будем использовать эти результаты для построения явных решений линейного дифференциального уравнения, включающего левостороннюю дробную производную  Бесселя на полуоси типа Герасимова--Капуто с постоянными коэффициентами.
	
	\begin{teo} Пусть $\alpha>0$. Преобразование Мейера $B_{\gamma,0+}^{-\alpha}f\in \mathscr{M}_{\gamma}^a(\mathbb{R}_+)$ имеет вид
		\begin{equation}\label{KTR1}
			\mathcal{K}_\gamma[(B_{\gamma,0+}^{-\alpha}\varphi)(x)](\xi)=\xi^{-2\alpha}\mathcal{K}_\gamma\varphi(\xi).
		\end{equation}
	\end{teo}
	\begin{proof} Начнем с \eqref{KTR1}. Пусть $g(x)=I_2^{0,\alpha}\varphi(x)$. Используя факторизацию \eqref{F0+}, получим
		$$
		\mathcal{K}_\gamma[(B_{\gamma,0+}^{-\alpha}\varphi)(x)](\xi)=\int\limits_{0}^{\infty} {k}_{\frac{\gamma-1}{2}} (x\xi)\,
		(B_{\gamma,0+}^{-\alpha}\varphi)(x)x^\gamma\,dx=
		$$
		$$=\frac{1}{2^{2\alpha}}\int\limits_{0}^{\infty} {k}_{\frac{\gamma-1}{2}} (x\xi)\,
		\,I_2^{\frac{\gamma-1}{2},\alpha}I_2^{0,\alpha}\varphi(x) x^{2\alpha+\gamma}\,dx=
		$$
		$$
		=\frac{1}{2^{2\alpha}}\int\limits_{0}^{\infty} {k}_{\frac{\gamma-1}{2}} (x\xi)\,
		\,I_2^{\frac{\gamma-1}{2},\alpha}g(x) x^{2\alpha+\gamma}\,dx=
		$$
		$$
		=\frac{1}{2^{2\alpha-1}\Gamma(\alpha)}\,\int\limits_{0}^{\infty} {k}_{\frac{\gamma-1}{2}} (x\xi)\,
		x\,dx\int\limits_0^x(x^2-u^2)^{\alpha-1}u^{\gamma}g(u)du=
		$$
		$$
		=\frac{1}{2^{2\alpha-1}\Gamma(\alpha)}\,\int\limits_{0}^{\infty}u^{\gamma}g(u)du\int\limits_u^\infty (x^2-u^2)^{\alpha-1}  {k}_{\frac{\gamma-1}{2}} (x\xi)\,x\,dx.
		$$
		Рассмотрим внутренний интеграл. Применяя формулу
		2.16.3.7 из \cite{IR2} вида
		\begin{equation}\label{ForK}
			\int\limits_a^\infty x^{1\pm\rho} (x^2-a^2)^{\beta-1} K_{\rho}(cx)dx=2^{\beta-1}a^{\beta\pm\rho}c^{-\beta}\Gamma(\beta)K_{\rho\pm\beta}(ac),\qquad a,c,\beta>0
		\end{equation}
		будем иметь
		$$
		\int\limits_u^\infty (x^2-u^2)^{\alpha-1}  {k}_{\frac{\gamma-1}{2}} (x\xi)\,x\,dx=\frac{2^{\frac{\gamma-1}{2}}\Gamma\left({\frac{\gamma+1}{2}} \right) }{\xi^{\frac{\gamma-1}{2}}}\int\limits_u^\infty (x^2-u^2)^{\alpha-1}  K_{\frac{\gamma-1}{2}} (x\xi)\,x^{1-\frac{\gamma-1}{2}}\,dx=
		$$
		$$
		=\frac{2^{\frac{\gamma-1}{2}}\Gamma\left({\frac{\gamma+1}{2}} \right) }{\xi^{\frac{\gamma-1}{2}}}\cdot 2^{\alpha-1}u^{\alpha-\frac{\gamma-1}{2}}\xi^{-\alpha}\Gamma(\alpha)K_{\frac{\gamma-1}{2}-\alpha}(u\xi)
		$$
		и
		$$
		\mathcal{K}_\gamma[(B_{\gamma,0+}^{-\alpha}\varphi)(x)](\xi)=\frac{2^{\frac{\gamma-1}{2}-\alpha}\Gamma\left({\frac{\gamma+1}{2}} \right) }{\xi^{\frac{\gamma-1}{2}+\alpha}}\int\limits_{0}^{\infty}u^{\alpha+\frac{\gamma+1}{2}}K_{\frac{\gamma-1}{2}-\alpha}(u\xi) g(u)du=
		$$
		$$
		=\frac{2^{\frac{\gamma+1}{2}-\alpha}\Gamma\left({\frac{\gamma+1}{2}} \right) }{\Gamma(\alpha)\xi^{\frac{\gamma-1}{2}+\alpha}}\,
		\int\limits_{0}^{\infty}u^{\frac{\gamma+1}{2}-\alpha}K_{\frac{\gamma-1}{2}-\alpha}(u\xi) du\int\limits_0^u(u^2-t^2)^{\alpha-1}t\varphi(t)dt=
		$$
		$$
		=\frac{2^{\frac{\gamma+1}{2}-\alpha}\Gamma\left({\frac{\gamma+1}{2}} \right) }{\Gamma(\alpha)\xi^{\frac{\gamma-1}{2}+\alpha}}\,
		\int\limits_{0}^{\infty}t\varphi(t)dt \int\limits_t^\infty (u^2-t^2)^{\alpha-1}u^{\frac{\gamma+1}{2}-\alpha}K_{\frac{\gamma-1}{2}-\alpha}(u\xi)du.
		$$
		Используя снова  \eqref{ForK}, запишем
		$$
		\int\limits_t^\infty (u^2-t^2)^{\alpha-1}u^{\frac{\gamma+1}{2}-\alpha}K_{\frac{\gamma-1}{2}-\alpha}(u\xi)du=
		2^{\alpha-1}t^{\frac{\gamma-1}{2}}\xi^{-\alpha}\Gamma(\alpha)K_{\frac{\gamma-1}{2}}(t\xi)
		$$
		и
		$$
		\mathcal{K}_\gamma[(B_{\gamma,0+}^{-\alpha}\varphi)(x)](\xi)=\frac{2^{\frac{\gamma+1}{2}-\alpha}\Gamma\left({\frac{\gamma+1}{2}} \right) }{\Gamma(\alpha)\xi^{\frac{\gamma-1}{2}+\alpha}}\,
		\cdot 2^{\alpha-1}\xi^{-\alpha}\Gamma(\alpha)\int\limits_{0}^{\infty}\varphi(t)K_{\frac{\gamma-1}{2}}(t\xi) t^{\frac{\gamma+1}{2}}dt =
		$$
		$$
		=\xi^{-2\alpha}\int\limits_{0}^{\infty}\varphi(t)k_{\frac{\gamma-1}{2}}(t\xi) t^{\gamma}dt=\xi^{-2\alpha}\mathcal{K}_\gamma\varphi.
		$$
		Доказательство закончено.
	\end{proof}
	
	\begin{lem}
		Пусть $n\in\mathbb{N}$  и $B_\gamma^n f\in \mathscr{M}_{\gamma}^a(\mathbb{R}_+)$, тогда для
		 $0\leq\gamma<1$
		$$
		\mathcal{K}_\gamma[B_\gamma^n f](\xi)=\xi^{2n}\mathcal{K}_\gamma[f](\xi)-
		$$
		\begin{equation}\label{MeijBesn01}
			-\sum\limits_{k=1}^{n}\xi^{2k-1-\gamma}B_\gamma^{n-k} f(0+)-\frac{\Gamma\left(\frac{1-\gamma}{2} \right) }{2^\gamma\Gamma\left(\frac{\gamma+1}{2} \right)}\lim\limits_{x\rightarrow 0+}\sum\limits_{k=1}^{n} \xi^{2k-2}
			x^\gamma \frac{d}{dx}[B_\gamma^{n-k} f(x)],
		\end{equation}
		для $\gamma=1$
		\begin{equation}\label{MeijBesn02}
			\mathcal{K}_\gamma[B_\gamma^n f](\xi)=
			\xi^{2n}\mathcal{K}_\gamma[f](\xi)-\sum\limits_{k=1}^{n}\xi^{2k-1-\gamma}B_\gamma^{n-k} f(0+)+\lim\limits_{x\rightarrow 0+}\sum\limits_{k=1}^{n}
			\xi^{2k-2}	\ln{x\xi}\frac{d}{dx}[B_\gamma^{n-k} f(x)],
		\end{equation}
		для $1<\gamma$
		$$
		\mathcal{K}_\gamma[B_\gamma^n f](\xi)=
		$$
		\begin{equation}\label{MeijBesn03}
			=\xi^{2n}\mathcal{K}_\gamma[f](\xi)-\sum\limits_{k=1}^{n}\xi^{2k-1-\gamma}B_\gamma^{n-k} f(0+)-\frac{1}{\gamma-1}\lim\limits_{x\rightarrow 0+}\sum\limits_{k=1}^{n} \xi^{2k-1-\gamma}
			x \frac{d}{dx}[B_\gamma^{n-k} f(x)],
		\end{equation}
		где
		$$
		B_\gamma^{n-k} f(0+)=\lim\limits_{x\rightarrow +0}B_{\gamma}^{n-k} f(x).
		$$
	\end{lem}
	\begin{proof}
		Найдем $\mathcal{K}_\gamma[B_\gamma^n f](\xi)$:
		$$
		\mathcal{K}_\gamma[B_\gamma^n f](\xi)=\int\limits_{0}^{\infty} {k}_{\frac{\gamma-1}{2}} (x\xi)\,
		[B_\gamma^n f(x)]\,x^\gamma\,dx
		=\int\limits_{0}^{\infty} {k}_{\frac{\gamma-1}{2}} (x\xi)\,
		\frac{d}{dx}x^\gamma \frac{d}{dx}[B_\gamma^{n-1} f(x)]\,dx=
		$$
		$$
		={k}_{\frac{\gamma-1}{2}} (x\xi)\,
		x^\gamma \frac{d}{dx}[B_\gamma^{n-1} f(x)]\biggr|_{x=0}^{\infty}-\int\limits_{0}^{\infty}x^\gamma\frac{d}{dx} {k}_{\frac{\gamma-1}{2}} (x\xi)\,
		\frac{d}{dx}[B_\gamma^{n-1} f(x)]\,dx=
		$$
		$$
		=-{k}_{\frac{\gamma-1}{2}} (x\xi)\,
		x^\gamma \frac{d}{dx}[B_\gamma^{n-1} f(x)]\biggr|_{x=0}+\left( x^\gamma \frac{d}{dx}{k}_{\frac{\gamma-1}{2}} (x\xi)\right) \,
		[B_\gamma^{n-1} f(x)]\biggr|_{x=0}+
		$$
		$$
		+\int\limits_{0}^{\infty} [B_\gamma {k}_{\frac{\gamma-1}{2}} (x\xi)]\,
		[B_\gamma^{n-1} f(x)]x^\gamma\,dx=-{k}_{\frac{\gamma-1}{2}} (x\xi)\,
		x^\gamma \frac{d}{dx}[B_\gamma^{n-1} f(x)]\biggr|_{x=0}+
		$$
		$$
		+\left( x^\gamma \frac{d}{dx}{k}_{\frac{\gamma-1}{2}} (x\xi)\right) \,
		[B_\gamma^{n-1} f(x)]\biggr|_{x=0}+\xi^2\int\limits_{0}^{\infty}  {k}_{\frac{\gamma-1}{2}} (x\xi)\,
		[B_\gamma^{n-1} f(x)]x^\gamma\,dx=...
		$$
		$$
		...=\xi^{2n}\int\limits_{0}^{\infty}  {k}_{\frac{\gamma-1}{2}} (x\xi)\,f(x)x^\gamma\,dx+
		$$
		$$
		+\sum\limits_{k=0}^{n-1}\xi^{2k}\left(\left( x^\gamma \frac{d}{dx}{k}_{\frac{\gamma-1}{2}} (x\xi)\right) \,
		[B_\gamma^{n-1-k} f(x)]-{k}_{\frac{\gamma-1}{2}} (x\xi)\,
		x^\gamma \frac{d}{dx}[B_\gamma^{n-1-k} f(x)]\right)\biggr|_{x=0}.
		$$
	Пусть $0\leq\gamma<1$, тогда используя \eqref{KBessAsO2}, получим
		$$
		\lim\limits_{x\rightarrow 0+}{k}_{\frac{\gamma-1}{2}} (x\xi)\,
		x^\gamma \frac{d}{dx}[B_\gamma^{n-1-k} f(x)]=\frac{\Gamma\left(\frac{1-\gamma}{2} \right) }{2^\gamma\Gamma\left(\frac{\gamma+1}{2} \right)}\lim\limits_{x\rightarrow 0+}
		x^\gamma \frac{d}{dx}[B_\gamma^{n-1-k} f(x)].
		$$
		Для $\gamma=1$, применяя \eqref{KBessAsO3}, будем иметь
		$$
		\lim\limits_{x\rightarrow 0+}{k}_{0} (x\xi)\,
		\frac{d}{dx}[B_\gamma^{n-1-k} f(x)]=-\lim\limits_{x\rightarrow 0+}
		\ln{x\xi}\frac{d}{dx}[B_\gamma^{n-1-k} f(x)].
		$$
	Когда $1<\gamma$, используя \eqref{KBessAsO1}, получим
		$$
		\lim\limits_{x\rightarrow 0+}{k}_{\frac{\gamma-1}{2}} (x\xi)\,
		x^\gamma \frac{d}{dx}[B_\gamma^{n-1-k} f(x)]=\frac{1}{\gamma-1}\lim\limits_{x\rightarrow 0+}
		x\xi^{1-\gamma} \frac{d}{dx}[B_\gamma^{n-1-k} f(x)].
		$$
Затем запишем
		$$
		\frac{d}{dx}{k}_{\frac{\gamma-1}{2}} (x\xi)=-\frac{2^{\frac{1-\gamma}{2}}\xi^{\frac{3-\gamma}{2}}
			x^{\frac{1-\gamma}{2}}}{\Gamma \left(\frac{\gamma +1}{2}\right)} K_{\frac{\gamma +1}{2}}(x \xi )
		$$
	и, применяя	 \eqref{AssimK} для близких к нулю $x$, будем иметь
		$$
		x^\gamma\frac{d}{dx}{k}_{\frac{\gamma-1}{2}} (x\xi)=-\frac{2^{\frac{1-\gamma}{2}}}{\Gamma \left(\frac{\gamma +1}{2}\right)} x^{\frac{\gamma+1}{2}}\xi^{\frac{3-\gamma }{2}} K_{\frac{\gamma +1}{2}}(x \xi )\sim
		$$
		$$
		\sim  -\frac{2^{\frac{1-\gamma}{2}}}{\Gamma \left(\frac{\gamma +1}{2}\right)} x^{\frac{\gamma+1}{2}}\xi^{\frac{3-\gamma }{2}}  \frac {\Gamma \left(\frac{\gamma +1}{2}\right)}{2^{1-\frac{\gamma +1}{2}}}(\xi x)^{-\frac{\gamma +1}{2}}
		=-\xi^{1-\gamma},\qquad x\rightarrow 0+,
		$$
	следовательно
		$$
		\lim\limits_{x\rightarrow 0+} \left( x^\gamma \frac{d}{dx}{k}_{\frac{\gamma-1}{2}} (x\xi)\right) \,
		[B_\gamma^{n-1-k} f(x)]=-\xi^{1-\gamma}B_\gamma^{n-1-k} f(0+)
		$$
		и для $0\leq\gamma<1$
		$$
		\mathcal{K}_\gamma[B_\gamma^n f](\xi)=\xi^{2n}\mathcal{K}_\gamma[f](\xi)-\sum\limits_{k=0}^{n-1}\xi^{2k+1-\gamma}B_\gamma^{n-1-k} f(0+)-
		$$
		$$-\frac{\Gamma\left(\frac{1-\gamma}{2} \right) }{2^\gamma\Gamma\left(\frac{\gamma+1}{2} \right)}\sum\limits_{k=0}^{n-1}\xi^{2k}\lim\limits_{x\rightarrow 0+}
		x^\gamma \frac{d}{dx}[B_\gamma^{n-1-k} f(x)]=
		$$
		$$
		=\xi^{2n}\mathcal{K}_\gamma[f](\xi)-\sum\limits_{k=1}^{n}\xi^{2k-1-\gamma}B_\gamma^{n-k} f(0+)-\frac{\Gamma\left(\frac{1-\gamma}{2} \right) }{2^\gamma\Gamma\left(\frac{\gamma+1}{2} \right)}\lim\limits_{x\rightarrow 0+}\sum\limits_{k=1}^{n} \xi^{2k-2}
		x^\gamma \frac{d}{dx}[B_\gamma^{n-k} f(x)],
		$$
	для $\gamma=1$
		$$
		\mathcal{K}_\gamma[B_\gamma^n f](\xi)=
		\xi^{2n}\mathcal{K}_\gamma[f](\xi)-\sum\limits_{k=1}^{n}\xi^{2k-1-\gamma}B_\gamma^{n-k} f(0+)+\lim\limits_{x\rightarrow 0+}\sum\limits_{k=1}^{n}
		\xi^{2k-2}\ln{x\xi}\frac{d}{dx}[B_\gamma^{n-k} f(x)],
		$$
		для $1<\gamma$
		$$
		\mathcal{K}_\gamma[B_\gamma^n f](\xi)=
		$$
		$$
		=\xi^{2n}\mathcal{K}_\gamma[f](\xi)-\sum\limits_{k=1}^{n}\xi^{2k-1-\gamma}B_\gamma^{n-k} f(0+)-\frac{1}{\gamma-1}\lim\limits_{x\rightarrow 0+}\sum\limits_{k=1}^{n} \xi^{2k-1-\gamma}
		x \frac{d}{dx}[B_\gamma^{n-k} f(x)].
		$$
	\end{proof}
	
	\begin{rem}
		Пусть $n\in\mathbb{N}$, $\frac{d}{dx}[B_\gamma^{n-k} f(x)]$ ограничено и $B_\gamma^n f\in \mathscr{M}_{\gamma}^a(\mathbb{R}_+)$ и $\gamma\neq 1$, тогда
		\begin{equation}\label{MeijBesn1}
			\mathcal{K}_\gamma[B_\gamma^n f](\xi)=\xi^{2n}\mathcal{K}_\gamma[f](\xi)-\sum\limits_{k=1}^{n}\xi^{2k-1-\gamma}B_\gamma^{n-k} f(0+).
		\end{equation}
		Если $\frac{d}{dx}[B_\gamma^{n-k} f(x)]\sim x^\eta$, $\eta>0$ при $x\rightarrow 0+$, то \eqref{MeijBesn1} справедливо для $\gamma=1$.
	\end{rem}
	
	\begin{rem} Поскольку $
		k_{-\frac{1}{2}}(x) =e^{-x},
		$
		то
		$$
		\mathcal{K}_0[f](\xi)=\mathcal{L}[f](\xi),
		$$
	где $\mathcal{L}[f]$ --- преобразование Лапласа функции $f$.
		Известно, что
		$$
		\mathcal{L}[f''](\xi)=\xi^2\mathcal{L}[f](\xi)-\xi f(0)-f'(0).
		$$
	С другой стороны,
		$$
		\frac{\Gamma\left(\frac{1-\gamma}{2} \right) }{2^\gamma\Gamma\left(\frac{\gamma+1}{2} \right)}\biggr|_{\gamma=0}=1,\qquad \sum\limits_{k=1}^{n}
		x^\gamma \frac{d}{dx}[B_\gamma^{n-k} f(x)]\biggr|_{\gamma=0,n=1}=f'(x)
		$$
		и
		$$
		\mathcal{K}_0[B_0 f](\xi)= Lf''(\xi)=\xi^2 \mathcal{K}_0[ f](\xi)-\xi f(0)-f'(0)=
		\mathcal{L}[f''](\xi).
		$$
		Аналогичная ситуация справедлива и для $\mathcal{K}_0[B_0^{n} f](\xi)$.
	\end{rem}
	
	\begin{teo}\label{teomej}
	Пусть $n=[\alpha]+1$ для нецелого  $\alpha$ и $n=\alpha$ для $\alpha\in \mathbb{N}$
	и  $\mathcal{B}_{\gamma,0+}^{\alpha} f\in \mathscr{M}_{\gamma}^a(\mathbb{R}_+)$, тогда
		
при $0\leq\gamma<1$
		$$
		\mathcal{K}_\gamma[\mathcal{B}_{\gamma,0+}^{\alpha} f](\xi)=
		$$
		\begin{equation}\label{MeiDr1}
			=	\xi^{2\alpha}\mathcal{K}_\gamma[f](\xi)
			-\sum\limits_{k=0}^{n-1}\xi^{2\alpha-2k-1-\gamma}B_\gamma^{k} f(0+)-\frac{\Gamma\left(\frac{1-\gamma}{2} \right) }{2^\gamma\Gamma\left(\frac{\gamma+1}{2} \right)}\lim\limits_{x\rightarrow 0+}\sum\limits_{k=0}^{n-1} \xi^{2\alpha-2k-2}
			x^\gamma \frac{d}{dx}[B_\gamma^{k} f(x)],
		\end{equation}
		
	при $\gamma=1$
		$$
		\mathcal{K}_\gamma[\mathcal{B}_{\gamma,0+}^{\alpha} f](\xi)
		$$
		\begin{equation}\label{MeiDr12}=\xi^{2\alpha}\mathcal{K}_\gamma[f](\xi)-\sum\limits_{k=0}^{n-1}\xi^{2\alpha-2k-1-\gamma}B_\gamma^{k} f(0+)+\lim\limits_{x\rightarrow 0+}\sum\limits_{k=0}^{n-1}
			\xi^{2\alpha-2k-2}\ln{x\xi}\frac{d}{dx}[B_\gamma^{k} f(x)],
		\end{equation}
		
при $1<\gamma$
		$$
		\mathcal{K}_\gamma[\mathcal{B}_{\gamma,0+}^{\alpha} f](\xi)=
		$$
		\begin{equation}\label{MeiDr3}=\xi^{2\alpha}\mathcal{K}_\gamma[f](\xi)-\sum\limits_{k=0}^{n-1}\xi^{2\alpha-2k-1-\gamma}B_\gamma^{k} f(0+)-\frac{1}{\gamma-1}\lim\limits_{x\rightarrow 0+}\sum\limits_{k=0}^{n-1} \xi^{2\alpha-2k-1-\gamma}
			x \frac{d}{dx}[B_\gamma^{k} f(x)],
		\end{equation}
	где
		$$
		B_{\gamma,0+}^{\alpha-k} f(0+)=\lim\limits_{x\rightarrow +0}B_{\gamma,0+}^{\alpha-k} f(x).
		$$
	\end{teo}
	\begin{proof} Используя   \eqref{KTR1} и \eqref{MeijBesn1} для $0\leq\gamma<1$, получим
		$$
		\mathcal{K}_\gamma[	\mathcal{B}_{\gamma,0+}^{\alpha} f](\xi)=\mathcal{K}_\gamma[(IB_{\gamma,0+}^{n-\alpha}B_\gamma^nf)(x)](\xi)
		=\xi^{2\alpha-2n}\mathcal{K}_\gamma[B_\gamma^nf](\xi)=\xi^{2\alpha}\mathcal{K}_\gamma[f](\xi)-
		$$
		$$
		-\sum\limits_{k=1}^{n}\xi^{2\alpha-2n+2k-1-\gamma}B_\gamma^{n-k} f(0+)-\frac{\Gamma\left(\frac{1-\gamma}{2} \right) }{2^\gamma\Gamma\left(\frac{\gamma+1}{2} \right)}\lim\limits_{x\rightarrow 0+}\sum\limits_{k=1}^{n} \xi^{2\alpha-2n+2k-2}
		x^\gamma \frac{d}{dx}[B_\gamma^{n-k} f(x)]=
		$$
		$$
		=\xi^{2\alpha}\mathcal{K}_\gamma[f](\xi)
		-\sum\limits_{k=0}^{n-1}\xi^{2\alpha-2k-1-\gamma}B_\gamma^{k} f(0+)-\frac{\Gamma\left(\frac{1-\gamma}{2} \right) }{2^\gamma\Gamma\left(\frac{\gamma+1}{2} \right)}\lim\limits_{x\rightarrow 0+}\sum\limits_{k=0}^{n-1} \xi^{2\alpha-2k-2}
		x^\gamma \frac{d}{dx}[B_\gamma^{k} f(x)],
		$$
	где	
		$$
	B_{\gamma,0+}^{k} f(0+)=\lim\limits_{x\rightarrow +0}B_{\gamma,0+}^{k} f(x).
		$$
	Аналогично при $\gamma=1$ будем иметь
		$$
		\mathcal{K}_\gamma[	\mathcal{B}_{\gamma,0+}^{\alpha} f](\xi)=\mathcal{K}_\gamma[(IB_{\gamma,0+}^{n-\alpha}B_\gamma^nf)(x)](\xi)
		=\xi^{2\alpha-2n}\mathcal{K}_\gamma[B_\gamma^nf](\xi)=\xi^{2\alpha}\mathcal{K}_\gamma[f](\xi)-
		$$
		$$
		-\sum\limits_{k=1}^{n}\xi^{2\alpha-2n+2k-1-\gamma}B_\gamma^{n-k} f(0+)+\lim\limits_{x\rightarrow 0+}\sum\limits_{k=1}^{n}
		\xi^{2\alpha-2n+2k-2}\ln{x\xi}\frac{d}{dx}[B_\gamma^{n-k} f(x)]=
		$$
		$$
		=\xi^{2\alpha}\mathcal{K}_\gamma[f](\xi)-\sum\limits_{k=0}^{n-1}\xi^{2\alpha-2k-1-\gamma}B_\gamma^{k} f(0+)+\lim\limits_{x\rightarrow 0+}\sum\limits_{k=0}^{n-1}
		\xi^{2\alpha-2k-2}\ln{x\xi}\frac{d}{dx}[B_\gamma^{k} f(x)]
		$$
	и при $\gamma>1$	
		$$
		\mathcal{K}_\gamma[	\mathcal{B}_{\gamma,0+}^{\alpha} f](\xi)=\mathcal{K}_\gamma[(IB_{\gamma,0+}^{n-\alpha}B_\gamma^nf)(x)](\xi)
		=\xi^{2\alpha-2n}\mathcal{K}_\gamma[B_\gamma^nf](\xi)=
		$$
		$$
		=\xi^{2\alpha}\mathcal{K}_\gamma[f](\xi)-\sum\limits_{k=0}^{n-1}\xi^{2\alpha-2k-1-\gamma}B_\gamma^{k} f(0+)-\frac{1}{\gamma-1}\lim\limits_{x\rightarrow 0+}\sum\limits_{k=0}^{n-1} \xi^{2\alpha-2k-1-\gamma}
		x \frac{d}{dx}[B_\gamma^{k} f(x)].
		$$	
	\end{proof}

	\begin{rem}\label{rem3}
		Пусть $k\in\mathbb{N}$, $\frac{d}{dx}[B_\gamma^{k} f(x)]$ ограничена,  $\mathcal{B}_{\gamma,0+}^{\alpha}f\in \mathscr{M}_{\gamma}^a(\mathbb{R}_+)$ и $\gamma\neq 1$, тогда
		\begin{equation}\label{MeijBesal1}
			\mathcal{K}_\gamma[	\mathcal{B}_{\gamma,0+}^{\alpha} f](\xi)=\xi^{2\alpha}\mathcal{K}_\gamma[f](\xi)
			-\sum\limits_{k=0}^{n-1}\xi^{2\alpha-2k-1-\gamma}B_\gamma^{k} f(0+).
		\end{equation}
		Если  $\frac{d}{dx}[B_\gamma^{k} f(x)]\sim x^\eta$, $\eta>0$ при $x\rightarrow 0+$, то \eqref{MeijBesal1} справедливо и для $\gamma=1$.
	\end{rem}

	\section{Метод преобразования Мейера для решения однородного  уравнения с левосторонней дробной производной Бесселя на полуоси типа Герасимова--Капуто}\label{Sec4}

\subsection{Общий случай}

Используя преобразование Мейера (общая схема применения интегральных преобразований к уравнениям дробного порядка  изложена в \cite{Luchko} и \cite{Thakur})	решим уравнение \begin{equation}\label{EQ}
	(\mathcal{B}_{\gamma,0+}^{\alpha} f)(x)=\lambda f(x),\qquad \alpha>0,\qquad \lambda\in\mathbb{R}
\end{equation}
с левосторонними дробными производными Бесселя на полуоси типа Герасимова--Капуто с постоянным коэффициентом при
 $\gamma\neq 1$.

Пусть $\frac{m-1}{2}<\alpha\leq\frac{m}{2}$, $m\in\mathbb{N}$. К уравнению \eqref{EQ} нужно добавить $m$ условий, которые для $0\leq \gamma<1$ имеют вид
\begin{equation}\label{EQCond}
	(B_{\gamma,0+}^{k} f)(0+)=a_{2k},\qquad \lim\limits_{x\rightarrow 0+}x^\gamma \frac{d}{dx}B_{\gamma,0+}^{k} f(x)=a_{2k+1},\qquad 	 a_{2k},a_{2k+1}\in\mathbb{R},
\end{equation}
а для $\gamma>1$ условия примет вид
\begin{equation}\label{EQCond2}
	(B_{\gamma,0+}^{k} f)(0+)=b_{2k},\qquad	\lim\limits_{x\rightarrow 0+}x \frac{d}{dx}B_{\gamma,0+}^{k} f(x)=b_{2k+1},\qquad 	
	b_{2k},b_{2k+1}\in\mathbb{R},
\end{equation}
где $k\in\mathbb{N}\cup\{0\}$, такое, что выполняются неравенства
$$
0\leq 2k\leq m-1, \qquad  1\leq 2k+1\leq m-2 \qquad \text{если}\,\, m\,\, \text{--- нечетное},
$$
и
$$
1\leq 2k+1\leq m-1, \qquad 0\leq 2k\leq m-2\qquad  \text{если}\,\, m\,\, \text{--- четное}.
$$
Это означает, что для нечетного $m$ последнее условие имеет вид
 $(B_{\gamma,0+}^{k} f)(0+){=}a_{m-1}$ или $(B_{\gamma,0+}^{k} f)(0+){=}b_{m-1}$,
а для четного $m$ последнее условие имеет вид $\lim\limits_{x\rightarrow 0+}x^\gamma \frac{d}{dx}B_{\gamma,0+}^{k} f(x){=}a_{m-1}$
или $\lim\limits_{x\rightarrow 0+}x \frac{d}{dx}B_{\gamma,0+}^{k} f(x){=}b_{m-1}$.

\begin{exampl} 
	
	При $m=1$  имеем, что $k=0$ и при $m=1$ к \eqref{EQ} добавляется только одно условие
		$$
f(0+)=a_0\quad \text{при}\quad 0\leq \gamma<1\quad \text{и}\quad f(0+)=b_0\quad \text{при}\quad \gamma>1. 
$$

		При $m=2$ также имеем, что $k=0$, но к \eqref{EQ} добавляются два условия
				$$
		f(0+)=a_0,\quad \lim\limits_{x\rightarrow 0+}x^\gamma \frac{df}{dx}=a_{1}\quad \text{при}\quad 0\leq \gamma<1\quad \text{и}
		$$
		$$
		f(0+)=b_0,\quad	\lim\limits_{x\rightarrow 0+}x \frac{df}{dx}=b_{1}\quad \text{при}\quad \gamma>1. 
		$$
		
\end{exampl}

\begin{teo} При $0\leq\gamma<1$ решение задачи \eqref{EQ}--\eqref{EQCond} 
	
в случае нечетного $m$ имеет вид
	$$
	f(x)=\frac{2^\gamma \Gamma\left(\frac{\gamma+1}{2}\right)}{\sqrt{\pi}}\sum\limits_{k=0}^{\frac{m-1}{2}}a_{2k}\,
	\,x^{2k} \,_2\Psi_2\left[\left.
	\begin{array}{c}
	$$\left(k+1+\frac{\gamma}{2},\alpha\right), (1,1)$$ \\
	$$\left(k+1,\alpha\right), \left(2k+\gamma+1,2\alpha\right)$$ \\
	\end{array}
	\right|\lambda x^{2 \alpha}\right]+
	$$
	\begin{equation}\label{Sol1}
		+	\frac{\Gamma\left(\frac{1-\gamma}{2} \right) }{\sqrt{\pi}}\sum\limits_{k=0}^{\frac{m-3}{2}}
		a_{2k+1}\,  \,x^{2k+1-\gamma}\,_2\Psi_2\left[\left.
		\begin{array}{c}
			$$\left(k+\frac{3}{2},\alpha\right), (1,1)$$ \\
			$$\left(k+\frac{3-\gamma}{2},\alpha\right), \left( 2k+2,2\alpha\right)$$ \\
		\end{array}
		\right|\lambda x^{2 \alpha}\right],
	\end{equation}
где вторая сумма исчезает при  $\frac{m-3}{2}<0$, то есть при $m=1$,
	
в случае четного $m$ имеет вид
	$$
	f(x)=
	\frac{2^\gamma \Gamma\left(\frac{\gamma+1}{2}\right)}{\sqrt{\pi}}\sum\limits_{k=0}^{\frac{m-2}{2}}a_{2k}\,
	\,x^{2k} \,_2\Psi_2\left[\left.
	\begin{array}{c}
	$$\left(k+1+\frac{\gamma}{2},\alpha\right), (1,1)$$ \\
	$$\left(k+1,\alpha\right), \left(2k+\gamma+1,2\alpha\right)$$ \\
	\end{array}
	\right|\lambda x^{2 \alpha}\right]+
	$$
	\begin{equation}\label{Sol2}
		+\frac{\Gamma\left(\frac{1-\gamma}{2} \right) }{\sqrt{\pi}}\sum\limits_{k=0}^{\frac{m-2}{2}}
		a_{2k+1}\,  \,x^{2k+1-\gamma}\,_2\Psi_2\left[\left.
		\begin{array}{c}
			$$\left(k+\frac{3}{2},\alpha\right), (1,1)$$ \\
			$$\left(k+\frac{3-\gamma}{2},\alpha\right), \left( 2k+2,2\alpha\right)$$ \\
		\end{array}
		\right|\lambda x^{2 \alpha}\right].
	\end{equation}
	
	Для $\gamma>1$ решение \eqref{EQ}--\eqref{EQCond2} 
	
в случае нечетного  $m$ имеет вид
	$$
	f(x)
	=
	\frac{2^\gamma \Gamma\left(\frac{\gamma+1}{2}\right)}{\sqrt{\pi}}\sum\limits_{k=0}^{\frac{m-1}{2}}b_{2k}\,
	\,x^{2k} \,_2\Psi_2\left[\left.
	\begin{array}{c}
	$$\left(k+1+\frac{\gamma}{2},\alpha\right), (1,1)$$ \\
	$$\left(k+1,\alpha\right), \left(2k+\gamma+1,2\alpha\right)$$ \\
	\end{array}
	\right|\lambda x^{2 \alpha}\right]+
	$$
	\begin{equation}\label{Sol3}
		+	\frac{2^\gamma \Gamma\left(\frac{\gamma+1}{2}\right)}{\sqrt{\pi}(\gamma-1)}
		\sum\limits_{k=0}^{\frac{m-3}{2}}
		b_{2k+1}\,  \,x^{2k} \,_2\Psi_2\left[\left.
		\begin{array}{c}
			$$\left(k+1+\frac{\gamma}{2},\alpha\right), (1,1)$$ \\
			$$\left(k+1,\alpha\right), \left(2k+\gamma+1,2\alpha\right)$$ \\
		\end{array}
		\right|\lambda x^{2 \alpha}\right],
	\end{equation}
где вторая сумма исчезает при  $\frac{m-3}{2}<0$, то есть при $m=1$,

в случае четного $m$ имеет вид
	$$
	f(x)=
	\frac{2^\gamma \Gamma\left(\frac{\gamma+1}{2}\right)}{\sqrt{\pi}}\sum\limits_{k=0}^{\frac{m-2}{2}}b_{2k}\,
	\,x^{2k} \,_2\Psi_2\left[\left.
	\begin{array}{c}
	$$\left(k+1+\frac{\gamma}{2},\alpha\right), (1,1)$$ \\
	$$\left(k+1,\alpha\right), \left(2k+\gamma+1,2\alpha\right)$$ \\
	\end{array}
	\right|\lambda x^{2 \alpha}\right]
	+
	$$
	\begin{equation}\label{Sol4}
		+\frac{2^\gamma \Gamma\left(\frac{\gamma+1}{2}\right)}{\sqrt{\pi}(\gamma-1)}
		\sum\limits_{k=0}^{\frac{m-2}{2}}
		b_{2k+1}\,  \,x^{2k} \,_2\Psi_2\left[\left.
		\begin{array}{c}
			$$\left(k+1+\frac{\gamma}{2},\alpha\right), (1,1)$$ \\
			$$\left(k+1,\alpha\right), \left(2k+\gamma+1,2\alpha\right)$$ \\
		\end{array}
		\right|\lambda x^{2 \alpha}\right].
	\end{equation}
	
Здесь  $_p\Psi_q(z)$ --- функция Райта--Фокса \eqref{Wright}.
\end{teo}
\begin{proof}Рассмотрим сначала случай  $0\leq\gamma<1$.
Применяя преобразование Мейера \eqref{MeT}
к обеим частям \eqref{EQ} и используя  \eqref{MeiDr1}, получим
	$$
	\xi^{2\alpha}\mathcal{K}_\gamma[f](\xi)
	-\sum\limits_{k=0}^{n-1}\xi^{2\alpha-2k-1-\gamma}B_\gamma^{k} f(0+)-\frac{\Gamma\left(\frac{1-\gamma}{2} \right) }{2^\gamma\Gamma\left(\frac{\gamma+1}{2} \right)}\lim\limits_{x\rightarrow 0+}\sum\limits_{k=0}^{n-1} \xi^{2\alpha-2k-2}
	x^\gamma \frac{d}{dx}[B_\gamma^{k} f(x)]= \lambda \mathcal{K}_\gamma[f](\xi),
	$$
где $n\in\mathbb{N}$, $n-1<\alpha\leq n$.
	Принимая во внимание условия \eqref{EQCond}, будем иметь
	
для случая нечетного $m$ будем иметь
	$$
	\xi^{2\alpha}\mathcal{K}_\gamma[f](\xi)	-\sum\limits_{k=0}^{\frac{m-1}{2}}a_{2k}\xi^{2\alpha-2k-1-\gamma}-\frac{\Gamma\left(\frac{1-\gamma}{2} \right) }{2^\gamma\Gamma\left(\frac{\gamma+1}{2} \right)}\sum\limits_{k=0}^{\frac{m-3}{2}}
	a_{2k+1}\xi^{2\alpha-2k-2}	= \lambda \mathcal{K}_\gamma[f](\xi),
	$$
где вторая сумма исчезает при  $\frac{m-3}{2}<0$, то есть при $m=1$,
	
для случая четного $m$ будем иметь
	$$
	\xi^{2\alpha}\mathcal{K}_\gamma[f](\xi)-\sum\limits_{k=0}^{\frac{m-2}{2}}a_{2k}\xi^{2\alpha-2k-1-\gamma}-\frac{\Gamma\left(\frac{1-\gamma}{2} \right) }{2^\gamma\Gamma\left(\frac{\gamma+1}{2} \right)}\sum\limits_{k=0}^{\frac{m-2}{2}}
	a_{2k-1}\xi^{2\alpha-2k-2}		= \lambda \mathcal{K}_\gamma[f](\xi).
	$$

Следовательно,
	
для случая нечетного $m$ будем иметь
	$$
	f(x)=\sum\limits_{k=0}^{\frac{m-1}{2}}a_{2k}\,
	\mathcal{K}_\gamma^{-1}\left[\frac{\xi^{2\alpha-2k-1-\gamma}}{\xi^{2\alpha}-\lambda}\right](x)+\frac{\Gamma\left(\frac{1-\gamma}{2} \right) }{2^\gamma\Gamma\left(\frac{\gamma+1}{2} \right)}\sum\limits_{k=0}^{\frac{m-3}{2}}
	a_{2k+1}\,\mathcal{K}_\gamma^{-1}\left[ \frac{\xi^{2\alpha-2k-2}}{\xi^{2\alpha}-\lambda}\right](x),
	$$
	
для случая четного $m$ будем иметь
	$$
	f(x)=\sum\limits_{k=0}^{\frac{m-2}{2}}a_{2k}\,\mathcal{K}_\gamma^{-1}\left[\frac{\xi^{2\alpha-2k-1-\gamma}}{\xi^{2\alpha}-\lambda}\right](x)+\frac{\Gamma\left(\frac{1-\gamma}{2} \right) }{2^\gamma\Gamma\left(\frac{\gamma+1}{2} \right)}\sum\limits_{k=0}^{\frac{m-2}{2}}
	a_{2k+1}\,\mathcal{K}_\gamma^{-1}\left[\frac{\xi^{2\alpha-2k-2}}{\xi^{2\alpha}-\lambda}\right](x).
	$$

	Для того чтобы найти явное выражение для $f$,  будем использовать формулу \eqref{InvK}. 
	Итак, сначала  найдем обратное преобразования Лапласа
	с учетом формулы  \eqref{LapOfML}:
	$$
	\mathcal{L}^{-1}\left[\frac{\xi^{2\alpha-2k-2}}{\xi^{2\alpha}-\lambda}\right](x)=x^{2k+1}E_{2\alpha,2k+2}(\lambda x^{2\alpha}),
	$$
	$$ \mathcal{L}^{-1}\left[\frac{\xi^{2\alpha-2k-1-\gamma}}{\xi^{2\alpha}-\lambda}\right](x)=x^{2k+\gamma}E_{2\alpha,2k+\gamma+1}(\lambda x^{2\alpha}).
	$$
Теперь найдем
 $$
	(\mathcal{P}_x^\gamma)^{-1}x^{\beta-\gamma}E_{2\alpha,\beta}(\lambda x^{2\alpha}).
	$$
	Используя \eqref{ObrPuass}, запишем
	$$
	(\mathcal{P}_x^\gamma)^{-1}x^{\beta-\gamma}E_{2\alpha,\beta}(\lambda x^{2\alpha})=\frac{2\sqrt{\pi}x}{\Gamma\left(\frac{\gamma+1}{2}\right)\Gamma \left( p-\frac{\gamma}{2} \right) } \left(\frac{d}{2xdx} \right)^{p} \int\limits_{0}^{x} z^{\beta}E_{2\alpha,\beta}(\lambda z^{2\alpha})  (x^2-z^2)^{p-\frac{\gamma}{2}-1} dz,
	$$
где $$
	p=\left[\frac{\gamma}{2}\right]+1.
	$$
Получим
	$$
	E_{2\alpha,\beta}(\lambda z^{2\alpha})=
	\sum_{m=0}^{\infty}\frac{\lambda^m z^{2m\alpha}}{\Gamma(2\alpha m+\beta)}
	$$
и
	$$
	\int\limits_{0}^{x} z^{\beta}E_{2\alpha,\beta}(\lambda z^{2\alpha})  (x^2-z^2)^{p-\frac{\gamma}{2}-1} dz=
	\sum_{m=0}^{\infty}\frac{\lambda^m}{\Gamma(2\alpha m+\beta)}\int\limits_{0}^{x} z^{2m\alpha+\beta}  (x^2-z^2)^{p-\frac{\gamma}{2}-1} dz=
	$$
	$$
	=\sum_{m=0}^{\infty}\frac{\lambda^m}{\Gamma(2\alpha m+\beta)}\,\frac{\Gamma \left( m\alpha+\frac{\beta+1}{2} \right)\Gamma \left( p-\frac{\gamma}{2} \right)}{2\Gamma \left( m\alpha+p+\frac{\beta-\gamma+1}{2} \right)}x^{2m\alpha+2p+\beta-\gamma-1}.
	$$
Следовательно,
	$$
	(\mathcal{P}_x^\gamma)^{-1}x^{\beta-\gamma}E_{2\alpha,\beta}(\lambda x^{2\alpha})=\frac{\sqrt{\pi}x}{\Gamma\left(\frac{\gamma+1}{2}\right)} \left(\frac{d}{2xdx} \right)^{p} \sum_{m=0}^{\infty}\frac{\lambda^m}{\Gamma(2\alpha m+\beta)}\,\frac{\Gamma \left( m\alpha+\frac{\beta+1}{2} \right)}{\Gamma \left( m\alpha+p+\frac{\beta-\gamma+1}{2} \right)}x^{2m\alpha+2p+\beta-\gamma-1}.
	$$
Используя формулу
	$$
	\left( \frac{d}{2xdx}\right)^nx^{2\mu+2n}=\frac{\Gamma(\mu+n+1)}{\Gamma(\mu+1)}x^{2\mu},
	$$
запишем
	$$
	(\mathcal{P}_x^\gamma)^{-1}x^{\beta-\gamma}E_{2\alpha,\beta}(\lambda x^{2\alpha})=\frac{\sqrt{\pi}x^{\beta-\gamma}}{\Gamma\left(\frac{\gamma+1}{2}\right)}  \sum_{m=0}^{\infty}\frac{\Gamma \left( m\alpha+\frac{\beta+1}{2} \right)}{\Gamma(2\alpha m+\beta)\Gamma \left( m\alpha+\frac{\beta-\gamma+1}{2} \right)}(\lambda x^{2\alpha})^m .
	$$
Принимая во внимание вид функции Фокса--Райта \eqref{Wright}, мы можем записать
	$$
	(\mathcal{P}_x^\gamma)^{-1}x^{\beta-\gamma}E_{2\alpha,\beta}(\lambda x^{2\alpha})=\frac{\sqrt{\pi}x^{\beta-\gamma}}{\Gamma\left(\frac{\gamma+1}{2}\right)}  \,_2\Psi_2\left[\left.
	\begin{array}{c}
	$$\left(\frac{\beta+1}{2},\alpha\right), (1,1)$$ \\
	$$\left(\frac{\beta-\gamma+1}{2},\alpha\right), \left( \beta,2\alpha\right)$$ \\
	\end{array}
	\right|\lambda x^{2 \alpha}\right] .
	$$
Тогда
	$$
	\mathcal{K}_\gamma^{-1}\left[ \frac{\xi^{2\alpha-2k-2}}{\xi^{2\alpha}-\lambda}\right](x)=\frac{1}{A_\gamma x}(\mathcal{P}_x^\gamma)^{-1}x^{1-\gamma}\left(\mathcal{L}^{-1}\left[ \frac{\xi^{2\alpha-2k-2}}{\xi^{2\alpha}-\lambda}\right]\right)(x)=
	$$
	$$
	=\frac{1}{A_\gamma x}(\mathcal{P}_x^\gamma)^{-1}x^{2k+2-\gamma}E_{2\alpha,2k+2}(\lambda x^{2\alpha})=
	$$
	$$
	=\frac{2^\gamma \Gamma\left(\frac{\gamma+1}{2}\right)}{\sqrt{\pi}}  \,x^{2k+1-\gamma}\,_2\Psi_2\left[\left.
	\begin{array}{c}
	$$\left(k+\frac{3}{2},\alpha\right), (1,1)$$ \\
	$$\left(k+\frac{3-\gamma}{2},\alpha\right), \left( 2k+2,2\alpha\right)$$ \\
	\end{array}
	\right|\lambda x^{2 \alpha}\right],
	$$
	\begin{equation}\label{Inv2}
		\mathcal{K}_\gamma^{-1}\left[\frac{\xi^{2\alpha-2k-1-\gamma}}{\xi^{2\alpha}-\lambda}\right](x)=\frac{2^\gamma \Gamma\left(\frac{\gamma+1}{2}\right)}{\sqrt{\pi}}  \,x^{2k} \,_2\Psi_2\left[\left.
		\begin{array}{c}
			$$\left(k+1+\frac{\gamma}{2},\alpha\right), (1,1)$$ \\
			$$\left(k+1,\alpha\right), \left(2k+\gamma+1,2\alpha\right)$$ \\
		\end{array}
		\right|\lambda x^{2 \alpha}\right].
	\end{equation}
	Таким образом, для случая нечетного $m$ получим \eqref{Sol1}, а для случая четного  $m$ --- \eqref{Sol2}.

Для $\gamma>1$, применяя преобразование Мейера \eqref{MeT} к обеим частям  \eqref{EQ} и используя  \eqref{MeiDr3}, получим
	$$
	\xi^{2\alpha}\mathcal{K}_\gamma[f](\xi)
	-\sum\limits_{k=0}^{n-1}\xi^{2\alpha-2k-1-\gamma}B_\gamma^{k} f(0+)-\frac{1}{\gamma-1}\lim\limits_{x\rightarrow 0+}\sum\limits_{k=0}^{n-1} \xi^{2\alpha-2k-1-\gamma}
	x \frac{d}{dx}[B_\gamma^{k} f(x)]= \lambda \mathcal{K}_\gamma[f](\xi),
	$$
где $n\in\mathbb{N}$, $n-1<\alpha\leq n$.
	Принимая во внимание условия \eqref{EQCond2}, получим

	для случая нечетного $m$ будем иметь
	$$
	f(x)=\sum\limits_{k=0}^{\frac{m-1}{2}}b_{2k}\,
	\mathcal{K}_\gamma^{-1}\left[\frac{\xi^{2\alpha-2k-1-\gamma}}{\xi^{2\alpha}-\lambda}\right](x)+\frac{1}{\gamma-1}\sum\limits_{k=0}^{\frac{m-3}{2}}
	b_{2k+1}\,\mathcal{K}_\gamma^{-1}\left[ \frac{\xi^{2\alpha-2k-1-\gamma}}{\xi^{2\alpha}-\lambda}\right](x),
	$$

для случая четного $m$ будем иметь
	$$
	f(x)=\sum\limits_{k=0}^{\frac{m-2}{2}}b_{2k}\,\mathcal{K}_\gamma^{-1}\left[\frac{\xi^{2\alpha-2k-1-\gamma}}{\xi^{2\alpha}-\lambda}\right](x)+
	\frac{1}{\gamma-1}\sum\limits_{k=0}^{\frac{m-2}{2}}
	b_{2k+1}\,\mathcal{K}_\gamma^{-1}\left[\frac{\xi^{2\alpha-2k-1-\gamma}}{\xi^{2\alpha}-\lambda}\right](x).
	$$
	
	Следовательно, применяя \eqref{Inv2}, получим \eqref{Sol3} и \eqref{Sol4}, соответственно.

\end{proof}

\subsection{Частные случаи и примеры}

В этом разделе сначала рассмотрим уравнение \eqref{EQ}, в случае, когда выполняются условия замечания \ref{rem3}. 
Затем приведем несколько примеров.

\begin{teo}\label{teo2}
	Пусть $m\in\mathbb{N}$,  $k\in\mathbb{N}\cup\{0\}$, $\frac{m-1}{2}<\alpha\leq\frac{m}{2}$, $\frac{d}{dx}[B_\gamma^{k} f(x)]$ ограничена для $0<\gamma$, $\gamma\neq 1$ и $\frac{d}{dx}[B_\gamma^{k} f(x)]\sim x^\beta$, $\beta>0$ при $x\rightarrow 0+$ в случае $\gamma=1$, тогда решение уравнения
	\begin{equation}\label{EQ2}
		(\mathcal{B}_{\gamma,0+}^{\alpha} f)(x)=\lambda f(x),\qquad \alpha>0,\qquad \lambda\in\mathbb{R}
	\end{equation}
	
с $m$ условиями для $0\leq \gamma<1$ вида
	\begin{equation}\label{EQCond5}
		(B_{\gamma,0+}^{k} f)(0+)=a_{2k},\qquad 	\lim\limits_{x\rightarrow 0+}x^\gamma \frac{d}{dx}B_{\gamma,0+}^{k} f(x)=0,	
	\end{equation}
	
с $m$ условиями для   $\gamma=1$ вида
	\begin{equation}\label{EQCond6}	
		(B_{\gamma,0+}^{k} f)(0+)=a_{2k},\qquad 	\lim\limits_{x\rightarrow 0+}	\ln{x}\frac{d}{dx}[B_\gamma^{k} f(x)]=0,	
	\end{equation}	
	
с $m$ условиями для   $\gamma>1$ вида
	\begin{equation}\label{EQCond7}
		(B_{\gamma,0+}^{k} f)(0+)=a_{2k},\qquad 	\lim\limits_{x\rightarrow 0+}x \frac{d}{dx}B_{\gamma,0+}^{k} f(x)=0,	
	\end{equation}
где $a_{2k}\in\mathbb{R}$ и
	$k$ такие, что следующие неравенства верны
	$$
	0\leq 2k\leq m-1, \qquad  1\leq 2k+1\leq m-2 \qquad \text{если}\,\, m\,\, \text{--- нечетное},
	$$
и
	$$
	1\leq 2k+1\leq m-1, \qquad 0\leq 2k\leq m-2\qquad  \text{если}\,\, m\,\, \text{--- четное}.
	$$
	Для случая нечетного $m$ решение примет вид 
	\begin{equation}\label{Sol5}
		f(x)=\frac{2^\gamma \Gamma\left(\frac{\gamma+1}{2}\right)}{\sqrt{\pi}}\sum\limits_{k=0}^{\frac{m-1}{2}}a_{2k}\,
		\,x^{2k} \,_2\Psi_2\left[\left.
		\begin{array}{c}
			$$\left(k+1+\frac{\gamma}{2},\alpha\right), (1,1)$$ \\
			$$\left(k+1,\alpha\right), \left(2k+\gamma+1,2\alpha\right)$$ \\
		\end{array}
		\right|\lambda x^{2 \alpha}\right],
	\end{equation}
а для случая четного $m$  решение примет вид
	\begin{equation}\label{Sol6}
		f(x)=\frac{2^\gamma \Gamma\left(\frac{\gamma+1}{2}\right)}{\sqrt{\pi}}\sum\limits_{k=0}^{\frac{m-2}{2}}a_{2k}\,
		\,x^{2k} \,_2\Psi_2\left[\left.
		\begin{array}{c}
			$$\left(k+1+\frac{\gamma}{2},\alpha\right), (1,1)$$ \\
			$$\left(k+1,\alpha\right), \left(2k+\gamma+1,2\alpha\right)$$ \\
		\end{array}
		\right|\lambda x^{2 \alpha}\right].
	\end{equation}
Здесь  $_p\Psi_q(z)$ --- функция Райта--Фокса \eqref{Wright}.
\end{teo}

\begin{exampl} Рассмотрим общий случай для задачи \eqref{EQ}--\eqref{EQCond} при $0<\alpha\leq\frac{1}{2}$, $0\leq\gamma<1$. В  этом случае $m=1$, $2k=0$
	используя \eqref{Sol1}, получим, что решение задачи
	$$
	(\mathcal{B}_{\gamma,0+}^{\alpha} f)(x)=\lambda f(x),\qquad \alpha>0,\qquad \lambda\in\mathbb{R},
	$$
	$$
	f(0+)=a_{0},\qquad a_{1}\in\mathbb{R}
	$$
имеет вид
	\begin{equation}\label{Sol7}
		f(x)=\frac{2^\gamma \Gamma\left(\frac{\gamma+1}{2}\right)}{\sqrt{\pi}}a_{0}\,
		\,_2\Psi_2\left[\left.
		\begin{array}{c}
			$$\left(1+\frac{\gamma}{2},\alpha\right), (1,1)$$ \\
			$$\left(1,\alpha\right), \left(\gamma+1,2\alpha\right)$$ \\
		\end{array}
		\right|\lambda x^{2 \alpha}\right].
	\end{equation}
Легко видеть, что для $\gamma>1$ и для $0<\alpha\leq\frac{1}{2}$ решение имеет такой же вид.	
На рисунке 1 представлен график  $f$ при $\gamma=\frac{1}{3}$ и при $\gamma=5$ когда $\alpha=\frac{1}{2}$, $\lambda=1$.	
	\begin{figure}[h!]
		\center{\includegraphics[width=0.7\linewidth]{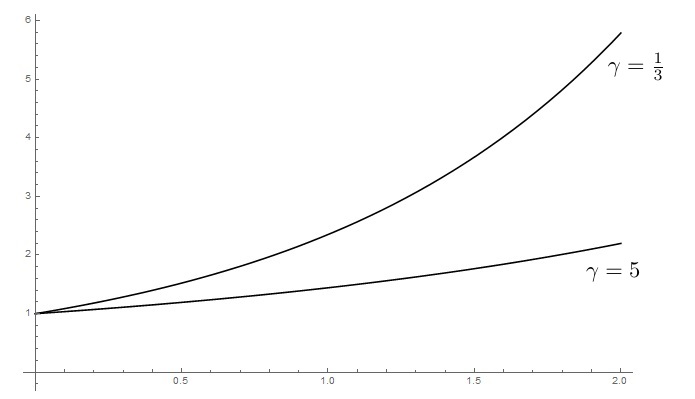} }
		\caption{График решения \eqref{Sol7} при $\gamma=\frac{1}{3}$ и при $\gamma=5$ когда $\alpha=\frac{1}{2}$, $\lambda=1$.}
	\end{figure}
	
	Когда $\gamma=0$ мы получим
	$$
	(\,^{GC}D_{0+}^{2\alpha} f)(x)=\lambda f(x),\qquad 0<2\alpha\leq 1,\qquad \lambda\in\mathbb{R},
	$$
	$$
	f(0+)=a_{1},\qquad a_{1}\in\mathbb{R}
	$$
	и, используя  \eqref{FWML}, запишем
	$$
	f(x)=a_{0} \,_2\Psi_2\left[\left.
	\begin{array}{c}
	$$\left(1,\alpha\right), (1,1)$$ \\
	$$\left(1,\alpha\right), \left(1,2\alpha\right)$$ \\
	\end{array}
	\right|\lambda x^{2 \alpha}\right]=a_{0} \,_1\Psi_1\left[\left.
	\begin{array}{c}
	$$ (1,1)$$ \\
	$$ \left(1,2\alpha\right)$$ \\
	\end{array}
	\right|\lambda x^{2 \alpha}\right]=a_{0}\, E_{2\alpha,1}(\lambda x^{2 \alpha}),
	$$
что совпадает с \eqref{SolCap} если $l=1$ и $2\alpha$ взято вместо $\alpha$.	
\end{exampl}

\begin{exampl}  Рассмотрим случай, представленный в Теореме \ref{teo2}  при $\alpha=1$, $b_{0}=1$, $\lambda=-1$. В этом случае $m=2$, $2k=0$, $2k+1=1$,
	что означает $k=0$. Используя \eqref{Sol6}, получим, что решение задачи
	$$
	{B}_{\gamma} f(x)=- f(x),\qquad  \lambda\in\mathbb{R},
	$$
	$$
	f(0+)=1,\qquad f'(0+)=0
	$$
	имеет вид
	$$
	f(x)=	\frac{2^\gamma \Gamma\left(\frac{\gamma+1}{2}\right)}{\sqrt{\pi}}
	\,_2\Psi_2\left[\left.
	\begin{array}{c}
	$$\left(1+\frac{\gamma}{2},1\right), (1,1)$$ \\
	$$\left(1,1\right), \left(\gamma+1,2\right)$$ \\
	\end{array}
	\right|-x^{2}\right]=\frac{2^\gamma \Gamma\left(\frac{\gamma+1}{2}\right)}{\sqrt{\pi}}
	\,_2\Psi_2\left[\left.
	\begin{array}{c}
	$$\left(1+\frac{\gamma}{2},1\right)$$ \\
	$$\left(\gamma+1,2\right)$$ \\
	\end{array}
	\right|-x^{2}\right]=
	$$
	$$
	=\frac{2^\gamma \Gamma\left(\frac{\gamma+1}{2}\right)}{\sqrt{\pi}}	\sum\limits_{m=0}^\infty \frac{(-1)^m\Gamma\left(1+\frac{\gamma}{2}+m\right)}{\Gamma\left( \gamma+1+2m\right) }
	\frac{x^{2m}}{m!}.
	$$
Использование формулы удвоения Лежандра вида
	$$
	\Gamma(2z)=\frac{2^{2z-1}}{\sqrt{\pi}}\Gamma(z)\Gamma\left(z+\frac{1}{2}\right)
	$$
получим
	$$
	f(x)=2^\gamma \Gamma\left(\frac{\gamma+1}{2}\right)
	\sum\limits_{m=0}^\infty \frac{(-1)^m\Gamma\left(1+\frac{\gamma}{2}+m\right)}{2^{\gamma+2m}\Gamma\left(1+\frac{\gamma}{2}+m\right)\Gamma\left(\frac{\gamma+1}{2}+m\right) }
	\frac{x^{2m}}{m!}=
	$$
	\begin{equation}\label{Sol8}
		=\frac{2^\frac{\gamma-1}{2} \Gamma\left(\frac{\gamma+1}{2}\right)}{x^\frac{\gamma-1}{2}}
		\sum\limits_{m=0}^\infty \frac{(-1)^m}{\Gamma\left(\frac{\gamma+1}{2}+m\right) }
		\frac{1}{m!}\left(\frac{x}{2}\right)^{2m+\frac{\gamma-1}{2}}= j_{\frac{\gamma-1}{2}}(x),
	\end{equation}
где
	$$
	j_\nu(x) =\frac{2^\nu\Gamma(\nu+1)}{x^\nu}\,\,J_\nu(x).
	$$
	Для $j_\frac{\gamma-1}{2}(x)$ имеем
	$$
	B_{\gamma} {j}_{\frac{\gamma-1}{2}}(\tau x)=-\tau^2{j}_{\frac{\gamma-1}{2}}(\tau x).
	$$
Следовательно, функция
	$$
	\,_2\Psi_2\left[\left.
	\begin{array}{c}
	$$\left(1+\frac{\gamma}{2},\alpha\right), (1,1)$$ \\
	$$\left(1,\alpha\right), \left(\gamma+1,2\alpha\right)$$ \\
	\end{array}
	\right|\lambda x^{2 \alpha}\right]
	$$
может быть рассмотрена как обобщение  $j_\frac{\gamma-1}{2}$.
\end{exampl}

\begin{exampl} Легко видеть, что при $\gamma=1$ оператор Бесселя есть двумерный оператор Лапласа в полярных координатах $x=r\cos\varphi$,
	$y=r\sin\varphi$, действующий на радиальную функцию $f=f(r)$:
	$$
	\Delta f=\frac{\partial^2 f}{\partial x^2}+\frac{\partial^2 f}{\partial y^2}=\frac{d^2f}{dr^2}+\frac{1}{r}\frac{df}{dr}.
	$$ 
Тогда уравнение \eqref{EQ} примет вид
\begin{equation}\label{rad}
\left(\frac{d^2}{dr^2}+\frac{1}{r}\frac{d}{dr} \right)^\alpha f(r)=\lambda f(r),
\end{equation}
где $\left(\frac{d^2}{dr^2}+\frac{1}{r}\frac{d}{dr} \right)^\alpha$ понимается в смысле определения \ref{def1}.

Пусть $m\in\mathbb{N}$,  $k\in\mathbb{N}\cup\{0\}$,   $\frac{m-1}{2}<\alpha\leq\frac{m}{2}$.
Если $\frac{d}{dr}[B_1^{k} f(r)]\sim r^\beta$, $\beta>0$ при $r\rightarrow 0+$, то к \eqref{rad} добавляются $m$ условий
	\begin{equation}\label{EQCond8}	
	\lim\limits_{r\rightarrow 0+}	\left(\frac{d^2}{dr^2}+\frac{1}{r}\frac{d}{dr} \right)^{k} f(r)=a_{2k},\qquad 	\lim\limits_{r\rightarrow 0+}	\ln{r}\frac{d}{dr}\left[\left(\frac{d^2}{dr^2}+\frac{1}{r}\frac{d}{dr} \right)^{k} f(r)\right]=0,	
\end{equation}	
	$k\in\mathbb{N}\cup\{0\}$, такие, что следующие неравенства верны
$$
0\leq 2k\leq m-1, \qquad  1\leq 2k+1\leq m-2 \qquad \text{если}\,\, m\,\, \text{--- нечетное},
$$
и
$$
1\leq 2k+1\leq m-1, \qquad 0\leq 2k\leq m-2\qquad  \text{если}\,\, m\,\, \text{--- четное}.
$$
 При $0<\alpha\leq\frac{1}{2}$ решение уравнения \eqref{rad} при условии
 $$
 \lim\limits_{r\rightarrow 0+} f(r)=b_0
 $$
 имеет вид
 \begin{equation}\label{Sol9}
 f(r)=\frac{2}{\sqrt{\pi}}b_{0}\,
 \,_2\Psi_2\left[\left.
 \begin{array}{c}
 $$\left(1+\frac{1}{2},\alpha\right), (1,1)$$ \\
 $$\left(1,\alpha\right), \left(2,2\alpha\right)$$ \\
 \end{array}
 \right|\lambda r^{2 \alpha}\right].
 \end{equation}
Решение для $\lambda=2$, $b_0=1$, $\alpha=0,1;0,3;0,5$ представлены на рисунке 2
	\begin{figure}[h!]
	\center{\includegraphics[width=1\linewidth]{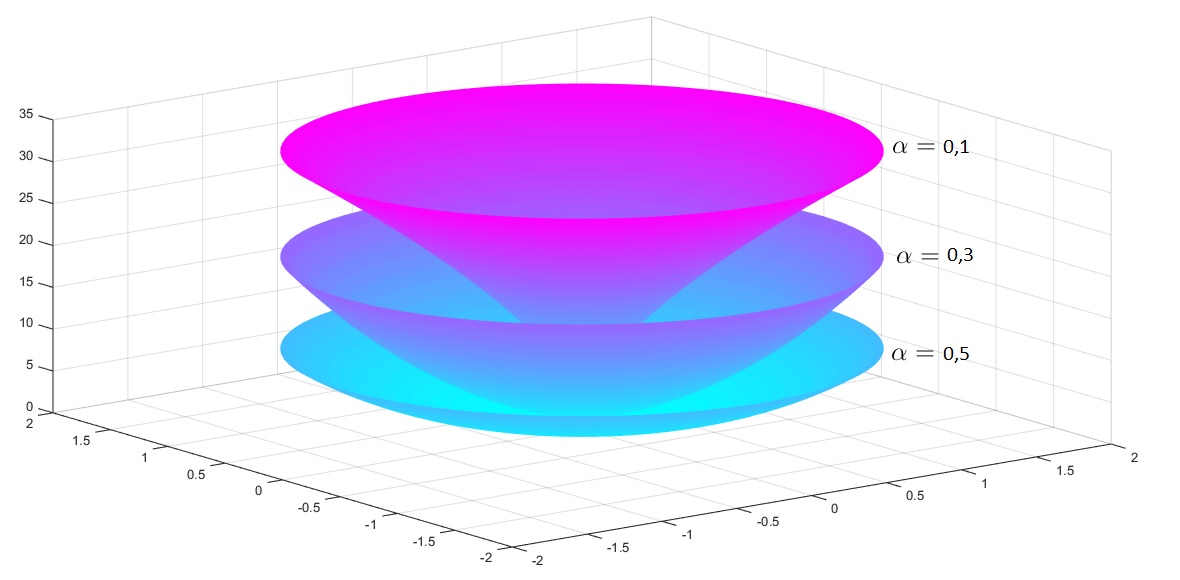} }
	\caption{График решения \eqref{Sol9} при  $\lambda=2$, $b_0=1$, $\alpha=0,1;0,3;0,5$.}
\end{figure}

\end{exampl}

\section{Заключение}

В данной статье предлагается новый подход для решения обыкновенных дифференциальных уравнений с левосторонней дробной производной Бесселя на полуоси типа Герасимова--Капуто на основе метод интегрального преобразования Мейера. Также приведено несколько иллюстративных примеров.

\end{document}